\documentclass[11pt]{amsart}

\usepackage{amssymb,amsfonts,amsthm,amsmath,mathrsfs,xspace,hyperref}
\usepackage{changes}
\definechangesauthor[name={Nico}, color=blue]{NM}
\usepackage{todonotes}
\usepackage{combelow}
\usepackage{graphicx}
\usepackage[french,english]{babel}
\usepackage{inputenc}
\usepackage[T1]{fontenc}
\usepackage[top=3cm, bottom=2cm, centering]{geometry}
\usepackage{csquotes}
\usepackage{color}
\usepackage{enumerate}
\usepackage{enumitem}

  \newcommand{\R}{\ensuremath{\mathbb{R}}}%
  \newcommand{\Z}{\ensuremath{\mathbb{Z}}}%
  \newcommand{\N}{\ensuremath{\mathbb{N}}}%

  \newcommand{\bim}{\mathfrak{B}}
  \newcommand{\arr}{\longrightarrow}
  \newcommand{\lims}[1][G]{\mathcal{J}_{#1}}
  \newcommand{\limg}[1][G]{\mathcal{X}_{#1}}
  \newcommand{\til}{\mathcal{T}}
  \newcommand{\xs}{\alb^\ast}
  \newcommand{\xo}{\alb^\omega}
  \newcommand{\leaf}{\mathcal{L}}
  \newcommand{\leaftilde}{\widetilde{\leaf}}
  \newcommand{\xmo}{\alb^{-\omega}}

        \newcommand{\M}{\ensuremath{\mathcal{M}}}%

  \newcommand{\Sym}{\mathop{\mathrm{Sym}}\nolimits}
  \newcommand{\Alt}{\mathop{\mathrm{Alt}}\nolimits}
    \newcommand{\acts}{\ensuremath{\curvearrowright}}%
  \newcommand{\sub}{\ensuremath{\operatorname{Sub}}}%
  \newcommand{\stab}{\ensuremath{\operatorname{Stab}}}%

    \newcommand{\aut}{\ensuremath{\operatorname{Aut}}}%
    \newcommand{\rist}{\ensuremath{\operatorname{RiSt}}}%
        \newcommand{\alb}{\ensuremath{\mathsf{X}}}%
        \newcommand{\nuke}{\ensuremath{\mathcal{N}}}%
        \newcommand{\Gammatilde}{\ensuremath{\widetilde{\Gamma}}}%

    \newcommand{\esf}{\ensuremath{\mathsf{e}}}

\theoremstyle{definition}
  \newtheorem{defin}{Definition}[section]

\theoremstyle{plain}
  \newtheorem{thm}[defin]{Theorem}
    \newtheorem*{thm*}{Theorem}

  \newtheorem{main thm}{Theorem}

  \newtheorem{prop}[defin]{Proposition}
    \newtheorem{prop-def}[defin]{Proposition-Definition}

  \newtheorem{cor}[defin]{Corollary}{}
    \newtheorem*{cor*}{Corollary}{}

  \newtheorem{lemma}[defin]{Lemma}

\theoremstyle{remark}
  \newtheorem{remark}[defin]{Remark}

\theoremstyle{plain}
    \newtheorem{introthm}{Theorem}

\newtheorem{introcor}[introthm]{Corollary}

\date{\today}	
\title{Commensurating actions and self-similar groups}

\author{Nicol\'as Matte Bon}
	\author{Volodymyr Nekrashevych}
	\author{Tianyi Zheng}

\thanks{NMB is partially supported by ANR project PLAGE ANR-24-CE40-3137. VN was supported by NSF grant DMS2204379. TZ is partially supported by NSF DMS 2348143.}
\begin{document}
\begin{abstract} Commensurating actions govern how a group can act on non-positively curved cube complexes. We obtain a complete picture of them for a class of finitely generated groups acting on rooted trees: contracting self-similar branch groups. The main application is a proof of Property FW for the iterated monodromy group of the subdivision rule generating the classical square Sierpi\'nski carpet.  This is the first example of an infinite finitely generated amenable group with Property FW, answering a question of Cornulier. As another application, we show that a contracting self-similar regular branch group does not have Property PW. In particular, the Grigorchuk group does not have Property PW, answering another question in the literature.

\end{abstract}

\maketitle

\section{Introduction}

\subsection*{Background and main application} 
A group $G$ \emph{commensurates} a subset $A$ of a $G$-set $X$ if $gA\triangle A$ is finite for all $g\in G$, and \emph{transfixes} $A$ if $A$ has finite symmetric difference with a $G$-invariant subset. For $G$ finitely generated, these notions admit a geometric interpretation in terms of ends of Schreier graphs: a $G$-set $X$ has a commensurated subset which is not transfixed if and only if some orbit in $X$ has a Schreier graph with more than one end (with respect to any fixed finite generating set). Commensurating actions govern how a group can act on $\mathrm{CAT}(0)$ cube complexes: a cellular action on such a complex determines a commensurating action (on the set of half-spaces of the complex), and conversely every commensurating action arises in this way, by a construction of Sageev \cite{Sageev}. Cornulier \cite{Cor-FM, Cor-FWsurvey} popularized the point of view of commensurating actions, recognizing in them a unifying framework for earlier work on cube complexes, median graphs, wall spaces (in the sense of \cite{Hag-Pau}), and related affine isometric actions on Hilbert spaces.

The group $G$ has \emph{Property FW} if every commensurated subset in every $G$-set is transfixed. Equivalently, every cellular action of $G$ on a $\mathrm{CAT}(0)$ cube complex has a fixed point. See \cite{Cor-FWsurvey} for more equivalent characterizations.

Cornulier asked whether there exist infinite finitely generated amenable groups with Property FW  \cite[Question 1.19(2)]{Cor-FM} (see also \cite[Question 1.11]{Cor-Piecewise}). 

In \cite{MNZ}, we established the amenability---via the Liouville property for random walk---of a class of \emph{contracting self-similar groups}: namely, all those whose \emph{limit space} has conformal dimension less than 2.  Among these, there is the \emph{Sierpi\'nski carpet group}---a finitely generated group encoding the subdivision rule generating the classical square Sierpi\'nski carpet (see \cite[\S6.3]{MNZ} and \S\ref{s-sierpinski} in this paper). Combining that result with the main theorems of the present paper, we obtain the following (see Theorem \ref{t-sierpinski}).

\begin{thm*}
The Sierpi\'nski carpet group is amenable and has Property FW.
\end{thm*}
This gives an affirmative answer to Cornulier’s question, filling in a missing piece in the picture of how commensurating actions relate to analytic group properties. Property FW can be viewed as a combinatorial weakening of Kazhdan's Property (T) \cite{NR, CV}. On the other hand, the opposite \emph{Property PW}---admitting a proper commensurating action (or equivalently, a proper action on a CAT(0) cube complex)---implies the Haagerup property \cite[Cor. 7.4.2]{CCJJA}. These implications have been widely used, for example, to disprove Property (T), or to establish the Haagerup property for many groups. They become equivalences if Property PW and FW are replaced by measured analogues \cite{RobertsonSteger, CDH}. The question whether every Haagerup group has Property PW had been circulating since the late 90s (see \cite{CN, CV} and \cite[\S2.C]{Cor-FWsurvey}); Haglund settled it in the negative in 2007\footnote{Although the paper \cite{Haglund} was published in 2023, the first preprint version is from 2007.} by showing that any group with Property PW must have all its infinite cyclic subgroups undistorted \cite{Haglund} (and many groups with the Haagerup property do not). Examples of groups that have Property FW but not (T), whose existence was also an open question (see \cite{Bar-Cha} and the second arXiv version of \cite{CDH}), were first exhibited by Cornulier \cite{Cor-FM}. Amenability is much more restrictive than the Haagerup property and is in strong opposition with Property (T), so the existence of amenable groups that have FW is quite a conclusive negative answer to the general question:

\begin{quote}
\emph{Can commensurating actions be characterized analytically?} 
\end{quote}

It was  natural to expect that infinite amenable groups with Property FW exist---Cornulier explicitly conjectures so just after \cite[Question 1.11]{Cor-Piecewise}. However, there are concrete obstructions to constructing such examples. An infinite finitely generated \emph{elementary amenable} group cannot have Property FW, as it admits a virtual surjection onto $\Z$; this forces one to look among the known non-elementary amenable groups---typically groups of dynamical origin. Amenability of essentially all known such examples can be established using Juschenko--Monod's \emph{extensive amenability} method \cite{Ju-Mo}. As it turns out, this method \emph{relies on commensurating actions}: amenability of a group $G$ is deduced from the existence of a $G$-set $X$ and a commensurated subset $A\subset X$ with amenable set-wise stabilizer, combined with the condition that the Schreier graphs of the action are recurrent \cite{JNS}. Subsequent criteria based on this approach \cite{JNS, JMMS} also implicitly depend on commensurating actions, as shown in \cite{LMT, Cor-Piecewise}. Any group covered by these criteria cannot have Property FW, thus ruling out essentially all known amenable groups (at least before \cite{MNZ}). Another difficulty is that methods for establishing Property FW beyond Property (T) are scarce (see \cite{Cor-FM, Cor-FWsurvey} and \cite[p.\,2 and \S4]{Genevois}). The only known criterion that does not immediately force non-amenability---derived from Haglund's work \cite{Haglund}, see \cite[Ex.\,6.A.8]{Cor-FM}---yields no known candidate amenable examples, even among groups whose amenability is open.

The main result in   \cite{MNZ} establishes amenability without \emph{a priori} relying on the existence of a commensurating action, thus producing new viable candidates to answer Cornulier's question. This motivated us to seek a precise understanding of commensurating actions of self-similar groups, and a tractable criterion for Property FW in this context. 
\subsection*{Main results} Our main results, Theorems~\ref{t-intro-germs} and~\ref{t-intro-cutpoints}, together provide a complete understanding of commensurating actions for a well-studied class of groups $G \leq \aut(T)$ acting on a rooted tree $T$, namely {contracting self-similar branch groups}. We refer to \S\ref{s-preliminaries-trees} and \S\ref{s-preliminaries-contracting} for the main definitions on branch groups and self-similar groups.

In many classical examples, multi-ended Schreier graphs appear as orbital graphs of the action on the boundary $\partial T$ of the tree, thereby precluding Property FW. This phenomenon occurs for the Grigorchuk group \cite{Grigorchuk:Schreier}, the Basilica group \cite{Basilica:Schreier}, and more generally for all \emph{bounded automata} groups. Another source of multi-ended Schreier graphs is given by \emph{graphs of germs}: for a boundary point $\xi \in \partial T$, the \emph{graph of germs} $\widetilde{\Gamma}_\xi$ is the Schreier graph associated to the coset space $G/G^0_\xi$, where $G^0_\xi \trianglelefteq G_\xi$ is the \emph{germ stabilizer}---the normal subgroup of the stabilizer $G_\xi$ consisting of elements fixing point-wise some neighborhood of $\xi$. The graph $\widetilde{\Gamma}_\xi$ is a Galois cover of the orbital graph $\Gamma_\xi$, with deck group equal to the quotient $G_\xi/G^0_\xi$, called the \emph{group of germs} at $\xi$.  In the case we are mostly interested about---contracting self-similar groups---the groups of germs of the action $G\acts \partial T$ are all {finite}.  Then each graph of germs $\Gammatilde_\xi$ is a finite cover of the orbital graph $\Gamma_\xi$ and has at least as many ends, but can have strictly more: for example the Grigorchuk group has a one-ended orbital graph $\Gamma_\xi$, covered by a four-ended graph of germs $\Gammatilde_\xi$ (see Vorobets~\cite{Vorobets}).

Our first main result asserts that, under natural assumptions, graphs of germs are essentially the only multi-ended Schreier graphs---up to finite covers. We refer the reader to \S\ref{s-preliminaries-trees} for the definition of a branch group. Recall that any proper quotient of a branch group $G$ is virtually abelian \cite{Gri-branch}, and that $G$ is \emph{just-infinite} if its proper quotients are all finite.

\begin{introthm} \label{t-intro-germs}
Let $G\le \aut(T)$ be a finitely generated branch group whose action on $\partial T$ has finite groups of germs. Let $H\le G$ be a subgroup such that the Schreier graph $\Gamma_{G/H}$ has more than one end. Then:
\begin{itemize}
    \item either the action of $G$ on $G/H$ is not faithful (and thus factors through a virtually abelian quotient); 
\item or $H\le G_\xi$ is a finite-index subgroup of the stabilizer of a point $\xi\in \partial T$. In this case, the number of ends of $\Gamma_{G/H}$ does not exceed that of the graph of germs $\Gammatilde_\xi$. 

\end{itemize} 
In particular, $G$ has Property FW if and only if it is just-infinite and all graphs of germs $\Gammatilde_
\xi, 
\xi\in \partial T$ are one-ended.
\end{introthm}
Our next result provides a conceptual explanation for the appearance of multi-ended graphs of germs in the case of \emph{contracting self-similar groups} (see \S\ref{s-preliminaries-contracting} for the definition). Such groups arise naturally as \emph{iterated monodromy groups} of locally expanding self-coverings of compact metric spaces (and more generally, of contracting virtual endomorphisms of orbispaces). Conversely, every contracting self-similar group $G$ has an associated \emph{limit space} $\lims$, a compact metrizable space equipped with an expanding branched covering $\mathsf{s} \colon \lims \to \lims$~\cite{Nek:book}. The space $\lims$ can be canonically uniformized as a quotient $\lims = \limg / G$, where $\limg$ is a locally compact space with a proper and co-compact right $G$-action, called the \emph{limit $G$-space}. This endows $\lims$ with the additional structure of an \emph{orbispace}. The branched covering $\mathsf{s}$ is uniformized as a virtual endomorphism of the orbispace $\lims$ (\emph{partial self-covering} in~\cite{Nek:book}). Then $G$ is the iterated monodromy group of $\mathsf{s}$. See more in~\cite{Nek:book} and~\cite[Chapter 4]{nek:dyngroups}. If $G$ is \emph{self-replicating}---a natural assumption, see \S\ref{s-preliminaries-contracting}---the spaces $\limg$ and $\lims$ are connected and locally path-connected.

There is a duality between the \emph{large-scale} geometry of the orbital graphs (respectively the graphs of germs) arising from the action of $G$ on the tree boundary, and the \emph{local} topology of the space $\lims$ (respectively $\limg$). The next theorem is a manifestation of this principle:

\begin{introthm} \label{t-intro-cutpoints}
Let $G \leq \aut(T)$ be a finitely generated, faithful contracting self-replicating group, with limit $G$-space $\limg$. Then the following quantities are finite and equal.
\begin{itemize}
    \item The supremum of the number of ends of the graph of germs $\Gammatilde_\xi$ for $\xi\in \partial T$.

    \item The supremum of the number of connected components of $U\setminus \{\zeta\}$ for $\zeta\in \limg$ and $U$ a connected neighbourhood of $\zeta$.
\end{itemize}
\end{introthm}
The combination of Theorems \ref{t-intro-germs} and \ref{t-intro-cutpoints} yields:
\begin{introcor} \label{c-intro-FW}
Let $G$ be a finitely generated, contracting self-replicating branch group.  Then $G$ has Property FW if and only if it is just-infinite and the limit $G$-space $\limg$ has no local cut point.
\end{introcor}
We stress that for many well-studied families of contracting groups, the space $\lims$ (and thus $\limg$) \emph{does} have local cut points. This is the case, in particular, when $G$ is generated by an automaton of polynomial activity (this is related to the fact that such groups fall in the setting of the extensive amenability criterion from~\cite{JNS}). On the other hand, the criterion of~\cite{MNZ} implies amenability for many groups for which $\limg$ is (locally) homeomorphic to the Sierpi\'nski carpet---and hence has no local cut point. Many such examples can be constructed using the theory of iterated monodromy groups; for instance, the Julia sets of many complex rational functions are homeomorphic to the Sierpi\'nski carpet.

To apply Corollary~\ref{c-intro-FW}, however, one must also verify that the resulting group $G$ is branch and just-infinite. While we lack a conceptual interpretation of these conditions in terms of limit spaces, the Sierpi\'nski carpet group $G$ (associated to the classical square realization of the Sierpi\'nski carpet) turns out to satisfy both properties, as we prove by explicit computation (see Proposition~\ref{pr:sierpinskibranch}). This fact came as a pleasant surprise.

As another application of Theorem~\ref{t-intro-germs}, we obtain an obstruction to Property PW. Recall that a group $G$ has \emph{Property PW} if it admits a \emph{proper} commensurating action, namely a commensurated subset $A \subset X$ (for some $G$-set $X$) such that the associated \emph{cardinal-definite function}
\[\nu_A(g)= |gA \triangle A|\] is proper (i.e., $\{g\in G\colon \nu_A(g)\leq n
\}$ is finite for all $n$). Equivalently, $G$ has an action on a $\mathrm{CAT}(0)$ cube complex which is metrically proper, meaning that the function $g \mapsto d(gx, x)$ is proper for some (and hence any) point $x$.

For a contracting self-similar group, cardinal-definite functions arising from graphs of germs on the tree boundary can be bounded by familiar quantities, see Proposition \ref{p-contracting-cardinal-definite}. Combined with that estimate,  Theorem \ref{t-intro-germs} (in fact the more precise Theorem \ref{t-multi-ended-graphs}) has the following corollary.  
\begin{introcor}\label{c-intro-not-PW}
Let $G$ be a finitely generated, contracting self-replicating group which is regular branch. Then $G$ does not have Property PW.
\end{introcor}
The first Grigorchuk group satisfies the assumptions in Corollary \ref{c-intro-not-PW} and thus does not have property PW (see Corollary \ref{c-Grigorchuk-PW}). This answers a problem stated in \cite[\S2.B]{Cor-FWsurvey}.

In contrast, Schneeberger~\cite{Schneeberger} proved that the Grigorchuk group admits an isometric action on a $\mathrm{CAT}(0)$ cube complex which is proper in the topological sense---properly discontinuous, which boils down to having finite vertex stabilizers \footnote{In fact, actions on cube complexes with finite stabilizers are just called \emph{proper} in \cite{Schneeberger}. This explains why the title of \cite{Schneeberger} seems in contradiction with Corollary \ref{c-intro-not-PW}.}. The action considered in \cite{Schneeberger} is obtained by applying the Sageev construction to a two-ended orbital graph of a specific point $\xi\in \partial T$. The action itself was known before (e.g.\ \cite[\S2.B]{Cor-FWsurvey}), but the fact that it has finite stabilizers was unexpected before~\cite{Schneeberger}.  Genevois' book \cite[\S14]{Genevois-book} contains an exposition of the same action, where it is also  checked that it is not metrically proper (this was left open in~\cite{Schneeberger}).

As pointed out by Cornulier (in~\cite[\S2.E]{Cor-FWsurvey}), the study of commensurating actions should not be reduced to the investigation of Properties PW and FW alone. A more ambitious goal is to obtain a complete classification, for groups of interest $G$, of all cardinal-definite functions on $G$, up to addition of bounded functions. 
Such classifications are currently available in some cases (e.g.\ for locally compact abelian groups, see \cite{Cor-FWsurvey}). In this paper, we adopt this perspective and spell out what our results say in terms of cardinal-definite functions (see Corollary~\ref{c-cardinal-definite} and Corollary \ref{c-cardinal-definite-contracting}).

\subsection*{Organization of the paper}
Section \ref{s-preliminaries} contains terminology and elementary preliminaries on Schreier graphs and  commensurating actions.

In Section \ref{s-products} we prove Theorem \ref{t-commensurating-product-discrete}, which provides a restriction on the transitive commensurating actions of a direct product $G=G_1\times G_2$ of finitely generated groups. This result, used in the proof of Theorem \ref{t-intro-germs}, might have some independent interest. 

In Section \ref{s-branch}, we deal with branch groups. After recalling the necessary preliminaries, we prove Theorem \ref{t-intro-germs} (in a more precise form, Theorem \ref{t-multi-ended-graphs}).

In Section \ref{s-cutpoints} we specialize the discussion to contracting self-similar groups. After recalling the necessary preliminaries, we prove Theorem \ref{t-intro-cutpoints} (in a more precise version, Theorem \ref{th:endsandcutpoints}). Then we give a general bound on the cardinal-definite functions arising from graphs of germs (Proposition \ref{p-contracting-cardinal-definite}), valid for contracting groups, that is used to derive Corollary \ref{c-intro-not-PW} (through the more precise Corollary \ref{c-cardinal-definite-contracting}).  In \S\ref{s-IMG}, we give an application by characterizing all iterated monodromy groups of post-critically finite rational functions $f\in \mathbb{C}(z)$ whose action on the tree boundary has one-ended graphs of germs (see Corollary \ref{c-IMG}). 

In Section \ref{s-sierpinski}, we consider the Sierpi\'nski carpet group $G$. We prove that it is just-infinite and branch and describe its limit $G$-space, which is naturally tiled by copies of the Sierpi\'nski carpet. We deduce from Corollary \ref{c-intro-FW} that it has Property FW.

\subsection*{Acknowledgments} We thank Yves Cornulier and Anthony Genevois for their useful comments.

\subsection*{Statement on AI use}
All of the mathematics in this paper was carried out by humans; no AI was involved in the proofs or in the search for ideas or references. Once the first draft was complete, we made limited use of AI to revise and improve the presentation:
\begin{itemize}
 \item Mistral AI
(CNRS subscription) helped drafting parts of the introduction;
\item Claude Opus~4.8 (Pro plan) was asked to produce a full referee report,
which helped us find some typos and minor inaccuracies (none affecting the core validity of proof).
\end{itemize}
The model's suggestions were verified and selectively adopted by the authors, who retain full responsibility for the paper's content. This statement is included
in accordance with the
\href{https://leidendeclaration.ai/}{Leiden Declaration on Artificial
Intelligence and Mathematics}.

\section{Preliminaries on Schreier graphs and commensurating actions} \label{s-preliminaries}
\subsection{Stabilizers, subgroups and Schreier graphs}

Let $G$ be a finitely generated group, endowed with a finite symmetric generating set $S$ (that will always be implicit and omitted from notation). The \emph{Schreier graph} of a transitive $G$-set $X$ is the labelled, oriented graph $\Gamma_X$ with vertex set $X$, and an edge $(x, sx)$ for every $x\in X$ and $s\in S$, labelled by $s\in S$. By definition, the Schreier graph of a subgroup $H\le G$ is the Schreier graph $\Gamma_{G/H}$ of its coset space. In this paper we consider only Schreier graphs of transitive $G$-sets; in particular all Schreier graphs are connected and (obviously) have degree bounded by $|S|$.  With slight abuse of notation, we will often write  $A\subset \Gamma_X$ instead of $A\subset X$ for a set of vertices of a Schreier graph (in particular for commensurated subsets).

Given $G$ acting on a set $X$, we use the notation $G_\xi$ for the stabilizer of a point  $\xi\in X$. 	We denote by $G(\xi)$ the orbit of $\xi$.	The \emph{orbital graph} of $\xi\in X$ is the Schreier graph of the $G$-set $G(
\xi)$, and is denoted by $\Gamma_\xi$. It is naturally isomorphic to $\Gamma_{G/G_
\xi}$. We will freely switch between these two points of view,  and think of the set of vertices of $\Gamma_\xi$ either as $G(
\xi)$, or as the coset space $G/G_\xi$.

If $X$ is a topological space, we also denote by  $G^0_\xi$ the \emph{germ stabilizer} of $\xi$, consisting of all elements $g\in G_\xi$ that fix point-wise a neighbourhood of $\xi$.  We have $G^0_\xi\unlhd G_\xi$, and the quotient $G_\xi/G^0_\xi$ is called the \emph{group of germs} at $\xi$. 

The coset space $G/G^0_\xi$ is naturally in correspondence with the set of {germs} of elements of $G$ at $\xi$. We denote the germ of $g\in G$ at a point $\xi\in X$ by $[g, \xi]$. It is the equivalence class of the pair $(g, \xi)$ with respect to the equivalence relation identifying $(g_1, \xi)$ with $(g_2, \xi)$ if there exists a neighborhood $U$ of $\xi$ such that $g_1|_U=g_2|_U$. We denote by $[G, \xi]=\{[g, \xi] \colon g\in G\}$ the set of all germs of elements of $G$ at $\xi$, which is a $G$-set with action $g\cdot [h, \xi]=[gh, \xi]$.
		
	The  Schreier graph $\Gamma_{G/G^0_\xi}$ is called the \emph{graph of germs} of $\xi$ and is denoted $\widetilde{\Gamma}_\xi$. Its set of vertices can be thought either as the coset space $G/G^0_\xi$, or as the set of germs $[G, \xi]$, and we will use both points of view. The natural map $G/G^0_\xi\to G/G_\xi$ induces a covering map $p\colon \Gammatilde_\xi\to \Gamma_\xi$,
which is a Galois covering with deck group $G_\xi/G^0_\xi$.

We recall that the space of subgroups $\sub(G)$ of a group $G$ is naturally a compact space when endowed with the \emph{Chabauty topology}, induced from the product topology on the set $\{0, 1\}^G$ of all subsets of $G$. When $G$ is finitely generated, and we fix a generating set $S$, the correspondence $H\mapsto (\Gamma_{G/H}, H)$ establishes a homeomorphic embedding of $\sub(G)$ into the space of  connected {rooted oriented $S$-labelled graphs}, consisting of all locally finite $S$-labelled oriented graphs $(\Gamma, x)$ with a preferred vertex.  The topology on the latter is defined by the metric $d((\Gamma_1, x_1),(\Gamma_2, x_2))=2^{-R},$ where $R>0$ is the largest integer such that the balls $B_{\Gamma_i}(x_i, R), i=1, 2$ are isomorphic as rooted $S$-labelled oriented graphs.

\subsection{Ends and commensurating actions}
        For a finite subset of vertices $\Sigma\subset \Gamma$ of a graph, we write $\pi_0(\Gamma\setminus \Sigma)$ for the set of connected components of $\Gamma\setminus \Sigma$. By definition, two points belong to the same component $U\in \pi_0(\Gamma\setminus \Sigma)$ if they can be connected by a path that avoids $\Sigma$. We further denote by $\pi_0(\Gamma\setminus \Sigma)_\infty$ the subset consisting of \emph{infinite} connected components. The \emph{number of ends} $\esf(\Gamma)\in \N\cup\{\infty\}$ is defined as \[\esf(\Gamma)=\sup_\Sigma |\pi_0(\Gamma\setminus \Sigma)_\infty|,\] where the supremum is taken over all finite subsets $\Sigma$. It is an invariant of quasi-isometry: in particular, if $\Gamma$ is a Schreier graph of a finitely generated group $G$, then $\esf(\Gamma)$ does not depend on the chosen generating set.

Recall from the introduction that given an action $G\acts X$,  a subset $A\subset X$ is commensurated if $|g(A)\triangle A|<\infty$ for every $g\in G$. In this case we call the function $\nu_A\colon G\to \N$ given by
$\nu_A(g):=|g(A)\triangle A|$
the \emph{cardinal-definite function} associated to $A$.\footnote{For context let us recall that a function $\ell$ on $G$ is cardinal-definite if and only if it has the form $\ell(g)=d(gx, x)$ for some action on a $\mathrm{CAT}(0)$-cube complex (equivalently a median graph), where $d$ is the $L^1$ metric. However, we will not use cube complexes, nor median graphs, in the paper.}
The subset $A$ is \emph{transfixed} if it has finite symmetric difference with a $G$-invariant subset.  A commensurated subset $A$ is transfixed if and only if the cardinal-definite function $\nu_A$ is bounded (Brailovsky, Pasechnik and Praeger \cite{BPP}, see also \cite[Theorem 4.E.1]{Cor-FWsurvey}).

A group $G$ has Property FW if and only if every commensurated subset in every $G$-set is transfixed. If a countable group $G$ has Property FW, then it is finitely generated \cite[Proposition 5.A.3(2)]{Cor-FWsurvey}.

A group $G$ has Property PW if and only if it admits a cardinal-definite function $\nu_A$ which is proper, meaning that $\{g\in G\colon \nu_A(g)\leq n
\}$ is finite for every $n\ge 0$.

 We say that two functions $f, h\colon G\to \R$ are at bounded distance if $f-h\in \ell^\infty(G)$.  As explained in \cite[\S2.E]{Cor-FWsurvey}, a conclusive goal in the study of commensurating actions of a group $G$  is a classification of all cardinal-definite functions of $G$ {up to bounded distance}. 

The following well-known properties of commensurating actions will be used without mention. We give the short proofs (when not tautological) for the convenience of the reader.

\begin{itemize}

\item A commensurated subset $A$ of a transitive $G$-set is transfixed if and only if $A$ or $A^c$ is finite. 
\item If $G$ is a finitely generated group, then a transitive $G$-set $X$ contains a commensurated subset which is not transfixed if and only if $\esf(\Gamma_X)\ge 2$. Indeed fix a finite generating set $S=S^{-1}$. If $A\subset X$ is commensurated and not transfixed, then $\partial A=\{a\in A\colon Sa\not \subset A\}$ is finite, and both $A$ and $A^c$ must contain infinite connected components of $\Gamma_X\setminus \partial A$, so that $|\pi_0(\Gamma_X\setminus \partial A)_\infty|\ge 2$. Conversely if $|\pi_0(\Gamma_X\setminus \Sigma)_\infty|\ge 2$ for some finite $\Sigma$, then any $A\in\pi_0(\Gamma_X\setminus \Sigma)_\infty$ is commensurated and not transfixed.

\item If $G$ is finitely generated and an action on $X$ commensurates a subset $A$, then for all but finitely many $G$-orbits $Y\subset X$, we have $Y\subset A$ or $Y\subset A^c$. (Informally, everything non-trivial happens in finitely many orbits). Indeed if $S=S^{-1}$ is a finite generating set, then for every orbit $Y$ such that $Y\cap A$ and $Y\cap A^c$ are non-empty, there is $s\in S$ and $a\in A\cap Y$ such that $s\cdot a\in Y\cap A^c$. Hence the number of such orbits $Y$ is at most $\sum_{s\in S} \nu_A(s)<\infty$.

\item As a consequence, if $G$ is finitely generated, then every unbounded cardinal-definite function $\nu_A$ on $G$ is at bounded distance from a sum $\sum_{i=1}^r \nu_{A_i}$, where each $A_i$ is a commensurated subset of a transitive $G$-set $X_i$  such that $\esf(\Gamma_{X_i})\ge 2$. 

\item In particular, a finitely generated group has Property FW if and only if every infinite Schreier graph of $G$ is one-ended.

\end{itemize}

We will need the following lemmas.

\begin{lemma}
\label{lem:endsofcoverings}
Let $G$ be a finitely generated group, and $K\le H\le G$ be subgroups such  that $[H:K]<\infty$. Then $\esf(\Gamma_{G/K})\ge \esf(\Gamma_{G/H})$. Furthermore, for every commensurated subset $A\subset G/H$, there exists a commensurated subset $\widetilde{A}\subset G/K$ such that $\nu_{\widetilde{A}}=[H: K] \nu_A$.

\end{lemma}
\begin{proof}
 The natural map $G/K\to G/H$ gives rise to a covering map $p\colon \Gamma_{G/K}\to \Gamma_{G/H}$, of degree $[H:K]$. Let $\Sigma\subset \Gamma_{G/H}$ be a finite subset. 
  Then $p^{-1}(\Sigma)$ is finite. The map $p$ induces a surjective map $p_\ast \colon\pi_0(\Gamma_{G/K}\setminus p^{-1}(\Sigma))\to \pi_0(\Gamma_{G/H}\setminus \Sigma)$.  Since $p$ has finite degree, infinite subsets have infinite image, and thus $p_\ast(\pi_0(\Gamma_{G/K}\setminus p^{-1}(\Sigma))_\infty)=\pi_0(\Gamma_{G/H}\setminus \Sigma)_\infty$. It follows that 
  $|\pi_0(\Gamma_{G/K}\setminus p^{-1}(\Sigma))_\infty|\ge |\pi_0(\Gamma_{G/H}\setminus \Sigma)_\infty|$. 
  
  Now suppose that $A\subset G/H$ is a commensurated subset. Then it is not difficult to see that $p^{-1}(A)$ is commensurated, and $\nu_{p^{-1}(A)}=[H:K]\nu_A$.  \qedhere
\end{proof}

\begin{lemma} \label{l-reduction-transitive}
	Let $G$ be a group and  $\alpha \colon G\acts X$  be transitive action having a commensurated subset $A\subset X$ which is not transfixed. Let $N$ be a finitely generated normal subgroup of $G$ such that $\alpha(N)$ is infinite. Then the following hold:
	\begin{enumerate}
	
	\item the action $\alpha \colon N \acts X$ has finitely many orbits;
	\item	 for every $N$-orbit $Y\subset X$ the sets $A\cap Y$ and $Y\setminus A$ are both infinite.\end{enumerate}	\end{lemma}
	
		\begin{proof}
	The group $G$ acts transitively on the set $\mathcal{P}$ of $N$-orbits, from which we deduce that all $N$-orbits have the same cardinality. This cardinality cannot be finite, since this would imply that $\alpha(N)$ is finite using that $N$ is finitely generated. Therefore all $N$-orbits are infinite.  We classify $N$-orbits $Y\in  \mathcal{P}$ into four different types according to whether the intersections $Y\cap A$ and $Y\cap A^c$ are finite  or infinite; say that $Y$ is of type $(\epsilon_A, \epsilon_{A^c})$ with $\epsilon_B\in \{f, \infty\}$, where $\epsilon_B=f$ (respectively $\epsilon_B=\infty$) if  the cardinality of $Y\cap B$ is finite (respectively infinite). Let $\mathcal{P}_{\epsilon_A, \epsilon_{A^c}}$ be the set of orbits of a given type.  The fact that $A$ is commensurated by $G$ easily implies that each set $\mathcal{P}_{\epsilon_A, \epsilon_{A^c}}$ is $G$-invariant, so exactly one of them is non-empty.  It is not possible that  $\mathcal{P}=\mathcal{P}_{f, f}$ since we already argued that $N$-orbits are infinite. Assume by contradiction that $\mathcal{P}=\mathcal{P}_{f,\infty}$. Note that since $N$ is finitely generated, there are only finitely many orbits such that both $Y\cap A$ and $Y\cap A^c$ are non-empty. Thus if $\mathcal{P}=\mathcal{P}_{f, \infty}$ we deduce that $A$ is finite, a contradiction. Similarly if $\mathcal{P}=\mathcal{P}_{\infty, f}$ then $A^c$ is finite, again a contradiction. It follows that $\mathcal{P}=\mathcal{P}_{\infty, \infty}$, proving the second item in the statement. It follows in particular that every $N$-orbit intersects both $A$ and $A^c$ and thus there are only finitely many $N$-orbits.		\end{proof}

\section{Commensurating actions of direct products} \label{s-products}
	The following result will be used in the proof of Theorem \ref{t-intro-germs}.
	\begin{thm} \label{t-commensurating-product-discrete}
	Let $G=G_1\times G_2$ where each $G_i$ is  finitely generated. Let $\alpha \colon G\acts X$ be a transitive action whose Schreier graph $\Gamma_X$ has at least two ends. Then  one of the following holds.
	\begin{enumerate}[label=(\roman*)]
	\item There exists $i=1, 2$ such that $\alpha(G_i)$ is finite. 
	\item   $\alpha(G)$ is virtually abelian.  
	\end{enumerate}
	\end{thm}
	
The main geometric argument of the proof is contained in the following lemma.
		\begin{lemma}\label{l-geometric-argument}
		Let $\Gamma$ be a connected locally finite graph with infinitely many ends, and $H$ be a finitely generated group of automorphisms of $\Gamma$ which acts transitively on $\Gamma$. Let $\Gamma_*$ be the graph obtained by adding to $\Gamma$ all edges of the Schreier graph of the action of $H$ on $\Gamma$ (with respect to some finite symmetric generating set of $H$). Then $\Gamma_*$ is one-ended.
\end{lemma}

\begin{proof}
Assume by contradiction that $\esf(\Gamma_*)\ge 2$. Then we can find a finite set of vertices $\Sigma$ such that  $\pi_0(\Gamma_*\setminus \Sigma)_\infty$ contains two distinct infinite components $B_1, B_2$. Each $B_i$ is a union of elements of  $\pi_0(\Gamma\setminus \Sigma)$, necessarily finitely many (since the whole $\pi_0(\Gamma\setminus \Sigma)$ is finite, as $\Gamma$ is connected and locally finite).  Since each $B_i$ is infinite, it follows that we can find  $A_i\in\pi_0( \Gamma \setminus \Sigma)_\infty$ contained in $B_i$ for $i=1, 2$. 
Then every $h\in H$ maps only finitely many points of $A_1$ inside $A_2$ and vice versa (indeed, if $h$ is a product of $k$ generators, then every point $x\in A_i$ such that $h(x)\in A_{i+1}$ must be the endpoint of a path of length $k$ in $\Gamma_\ast$ passing through $\Sigma$, so all such points are in $B_{\Gamma_\ast}(\Sigma, k)$). Set $n=|\Sigma|$. Since $\Gamma$ has infinitely many ends, we can find a finite set $\Xi$ such that $|\pi_0(\Gamma \setminus \Xi)_\infty|\ge 2n+1$. We claim that there exists $g_1, g_2\in H$ such that $g_i(\Xi)$ is contained in $A_i$ for $i=1, 2$. To see this, fix $v_0\in \Gamma$ and let $R>0$ be such that $\Xi$ is contained in the ball $B_\Gamma(v_0, R)$.  For $i=1, 2$, pick $v_i\in A_i$ whose distance from $\Sigma$ in $\Gamma_*$ is $>R$. Then any choice of $g_i\in H$  such that $g_i(v_0)=v_i$ must satisfy $g_i(B_\Gamma(v_0, R))\subset A_i$, and thus $g_i(\Xi)\subset A_i$, showing the claim. Now set $\Xi_i:= g_i(\Xi)$ and note that if $Y\in\pi_0(\Gamma \setminus \Xi_i)$  does not contain any point in $\Sigma$, then $Y\subset A_i$: indeed some point of $Y$ is adjacent to a point $x\in \Xi_i$, and can be connected to every other point of $Y$ by a path in $\Gamma$ avoiding $\Sigma$, thus all points of $Y$ are in the component of $\Gamma\setminus \Sigma$ containing $x$, namely $A_i$. The element $h=g_2g_1^{-1}$ maps bijectively  $\pi_0(\Gamma \setminus \Xi_1)_\infty$ to  $\pi_0(\Gamma \setminus \Xi_2)_\infty$. Since both sets have at least $2n+1$ elements, and $|\Sigma|\le n$, we deduce that there must exist an infinite connected component $Y$ of $\Gamma \setminus \Xi_1$ such that $Y\subset A_1$ and  $h(Y)\subset A_2$. This contradicts that $h$ maps only finitely many points of $A_1$ inside $A_2$. \qedhere 
\end{proof}

\begin{proof}[Proof of Theorem \ref{t-commensurating-product-discrete}]
Let $\alpha\colon G\acts X$ be as in the statement, and $A\subset X$ be a commensurated subset that is not transfixed. Assume that $\alpha(G_i)$ is infinite for $i=1, 2$. Fix a  $G_1$-orbit  $Y\subset X$, and let $H_2$ be the stabilizer of $Y$ in $G_2$. By Lemma \ref{l-reduction-transitive} applied to $N=G_1$, both $Y\cap A$ and $Y\cap A^c$ are infinite, and $H_2$ has finite index in $G_2$. Note that the image of the action of $H_2$ on $Y$ cannot be finite, since otherwise $G_2$ would have a finite orbit, and thus $\alpha(G_2)$ would be finite by Lemma \ref{l-reduction-transitive}. Applying again Lemma \ref{l-reduction-transitive}  to the action of $G_1\times H_2$ on $Y$ and the subset $A\cap Y$, we obtain that for every $H_2$-orbit $Z\subset Y$ the intersections $Z\cap A$ and $Z\cap A^c$ are both infinite. Fix an $H_2$-orbit $Z\subset Y$ and let $H_1$ be the stabilizer of $Z$ in $G_1$. The same reasoning as above shows that $H_1$ has finite index in $G_1$. Moreover since $G_1$ acts transitively on $Y$ and the $H_2$-orbits form a $G_1$-invariant partition, we have that $H_1$ must act transitively on $Z$. Note that since $H_2$ acts transitively on $Z$ and its action commutes with $H_1$, the point stabilizers for the action of $H_1$ on $Z$ must coincide, and thus must be all equal to some given normal subgroup $N_1$ of $H_1$, and similarly the point stabilizers of the action of $H_2$ on $Z$ must be all equal to some normal subgroup $N_2$ of $H_2$.

Now for each $i=1, 2$ let $\Gamma_i$ be the Schreier graph of the action of $H_i$ on $Z$ (which coincides with the Cayley graph of $H_i/N_i$). Since both actions commensurate $A\cap Z$, each graph $\Gamma_i$ has more than one end. Moreover $H_1$ is a transitive group of automorphisms of $\Gamma_2$ and vice versa. Thus $\Gamma_1$ and $\Gamma_2$  have either 2 or infinitely many ends. 

Assume by contradiction that $\Gamma_1$ has infinitely many ends. We apply Lemma \ref{l-geometric-argument} to the action of $H_2$ on $\Gamma_1$. Note that the graph denoted $(\Gamma_1)_*$ in Lemma \ref{l-geometric-argument} is nothing but the Schreier graph of the action of $H_1\times H_2$ on $Z$, and since this action commensurates $A\cap Z$ the graph $(\Gamma_1)_*$ must have more than one end. This is a contradiction, and we deduce that $\Gamma_1$ must have two ends, and thus $H_1/N_1$ is virtually cyclic, by the classical Freudenthal--Hopf theorem \cite{Freudenthal, Hopf}. The same reasoning shows that $H_2/N_2$ is virtually cyclic. 

Now recall that $Y\subset X$ was an arbitrary $G_1$-orbit. The $H_2$-orbits $Z\subset Y$ form a finite $G_1$-invariant partition $\mathcal{P}$. We have just argued that the stabilizer of each $Z\in \mathcal{P}$ acts on $Z$ via a virtually cyclic quotient. This implies that $G_1$ acts on $Y$ via a virtually abelian quotient. Since there are finitely many $G_1$-orbits $Y$, this implies that $\alpha(G_1)$ is virtually abelian. The symmetric argument proves that $\alpha(G_2)$ is virtually abelian, and so $\alpha(G)$ is virtually abelian. \qedhere\end{proof}

\section{Branch groups} \label{s-branch}

\subsection{Definitions and preliminaries}

 \label{s-preliminaries-trees}
 Let $T$ be a locally finite rooted tree. We denote by $T^{(n)}$ the set of vertices at distance $n$ from the root. We shall always assume that $T$ is \emph{spherically homogeneous}, i.e. that all vertices in $T^{(n)}$ have the same degree. For a vertex $v\in T$, we write $|v|$ for its distance from the root, and  $T_v\subset T$ the subtree rooted at $v$, consisting of all vertices $w$ such that the geodesic from the root to $w$ passes through $v$. The boundary $\partial T$ is the set of geodesic rays starting at the root; it is naturally a compact space.

 We denote by $\aut(T)$ the group of automorphisms of $T$ that fix the root.  For $G\le \aut(T)$, we denote by $G_v$ the stabilizer of $v\in T$, and by $G_v^0$ its normal subgroup
\[G_v^0=\{g\in G_v \colon g|_{T_v}=\mathsf{id}_{T_v}\}.\]

For a boundary point $\xi\in \partial T$, the notations $G_\xi$ and $G^0_\xi$ stand for the stabilizer and the germ-stabilizer (introduced in \S\ref{s-preliminaries}). If $v_n\in T$ is the ordered sequence of vertices of a ray $\xi\in \partial T$, we have   $G_{v_n}\ge G_{v_{n+1}}$ and $G_{v_n}^0\le G_{v_{n+1}}^0$, and
\begin{equation}
    \label{eq:stabilizerslimit}
    G_\xi=\bigcap_{n\ge 0}G_{v_n},\qquad G_\xi^0=\bigcup_{n\ge 0}G_{v_n}^0.
\end{equation}

 The group $G\le \aut(T)$ is \emph{level-transitive} if its action on $T^{(n)}$ is transitive for every $n$. \begin{defin} \label{d-graphs-trees}
     Let $G\le \aut(T)$ be a level-transitive subgroup. \begin{enumerate} 
     
     \item We denote by $\Gamma_n$ the Schreier graph of its action on $T^{(n)}$, i.e. $\Gamma_n=\Gamma_{G/G_v}$ for any $v\in T^{(n)}$.

     \item We denote by $\Xi_n$ the Schreier graph $\Gamma_{G/G^0_v}$, for $v\in T^{(n)}$.
     \end{enumerate}
 \end{defin}

By level-transitivity, the subgroups $G^{0}_v$ are all conjugate for $v\in T^{(n)}$, and so the graph $\Xi_n$ does not depend on the choice of $v$. Note that the coset space $G/G^0_v$ is naturally identified with the set of restrictions $\{g|_{T_v}\colon g\in G\}$ for any $v\in T^{(n)}$.

For a ray $\xi=(v_n)_n\in \partial T$, the inclusions $G_\xi\le G_{v_n}$ and $G_{v_n}^0\le G_{\xi}^0$ induce covering maps of labeled directed graphs $\Gamma_\xi\arr\Gamma_n$ and $\Xi_n\arr\Gammatilde_{\xi}$.It follows from~\eqref{eq:stabilizerslimit} that for any ray $\xi=(v_n)\in \partial T$,  we have the convergences $\lim_n G^0_{v_n}=G^0_\xi$ and $\lim_n G_{v_n}=G_\xi$ in $\sub(G)$. As a consequence we have the following.

\begin{prop}
\label{pr:convergenceofgraphs}
 Let $\xi=(v_n)_n\in\partial T$. Then for every $R>0$ there exists $n$ such that the covering maps $\Gamma_\xi\arr\Gamma_n$ and $\Xi_n\arr\Gammatilde_\xi$ induce isomorphisms $B_{\Gamma_\xi}(\xi, R)\arr B_{\Gamma_n}(v_n, R)$ and $B_{\Xi_n}(G^0_{v_n}, R)\arr B_{\Gammatilde_\xi}([e,\xi], R)$.
\end{prop}

 When a sequence of vertices $(v_n)$ converges to a boundary point, but not necessarily along a ray, we still have the following semi-continuity.
\begin{lemma}\label{l-semi-continuous}
    Let $G\le \aut(T)$. Let $(v_n)\subset T$ be any sequence of vertices converging to $\xi\in \partial T$ (in the topology of $T\cup \partial T$), and suppose that $G^0_{v_n}$ converges to $H\in \sub(G)$. Then $G^0_\xi\le H\le G_\xi$. 
\end{lemma}
\begin{proof}
It is obvious that every element $g\in G^0_\xi$ satisfies $g\in G^0_{v_n}$ for all large enough $n$, hence $g\in H$. Any $g\in H$ satisfies $g(v_n)=v_n$ for all large enough $n$, hence $g(\xi)=\xi$. \qedhere

\end{proof}

 The \emph{rigid stabilizer} of $v\in T$ is the subgroup $\rist(v)$ of the stabilizer $G_v$ consisting of elements that fix every vertex of $T\setminus T_v$. The $n$th level rigid stabilizer is the subgroup $\rist(n)=\langle \rist(v) \colon |v|=n\rangle\simeq \prod_{|v|=n}\rist (v)$. 

  Let us recall the classical ``double commutator lemma'' in the special case of groups acting on rooted trees, see Grigorchuk \cite{Gri-branch}. We write $G'$ for the commutator subgroup of a group $G$. 

 \begin{lemma}  \label{l-double-comm}
     Let $G\le \aut(T)$. Let $g\in G$ and $v\in T$ be such that $g(v)\neq v$. Then the normal closure of $g$ in $G$ contains $\rist(v)'$. 
     In particular, if $G$ is level-transitive, then every non-trivial normal subgroup of $G$ contains $\rist(n)'$ for some $n$. 
 \end{lemma}

 \begin{defin}
     
 A subgroup $G\le \aut(T)$  is \emph{branch} if it is level-transitive and $\rist(n)$ has finite index in $G$ for every $n$. 

  \end{defin}

Note that in a finitely generated branch group $G$, all the subgroups $
\rist(v)$ are finitely generated (as they are direct factors of a finite index subgroup of $G$). 

For branch groups,  Lemma \ref{l-double-comm} has the following consequence. (Recall that a group $G$ is \emph{just-infinite} if every proper quotient of $G$ is finite.)

\begin{prop}[\cite{Gri-branch}] \label{p-Grigorchuk}
Let $G\le \aut(T)$ be a branch group. Then every proper quotient of $G$ is virtually abelian. 
Furthermore, $G$ is just-infinite if and only if for every $v\in T$ the commutator subgroup $\rist(v)'$ has finite index in $\rist(v)$. 
\end{prop}

The following lemma is a slightly more precise version of Proposition \ref{p-Grigorchuk}, that also applies to normal subgroups of the rigid stabilizers $\rist(v)$. (The latter are not necessarily branch, because they may not act level-transitively on $\partial T_v$.) 

\begin{lemma} \label{l-quotient-rist}
    Let $G\le \aut(T)$ be a branch group, and fix $v\in T$. Then for every normal subgroup $N\unlhd \rist(v)$, there is $n\ge 1$ such that for every  $w\in T^{(n)}_v$, either $N$ contains $\rist(w)'$, or $N$ fixes $w$ and acts trivially on $T_w$.

    In particular $\rist(v)/N$ is virtually abelian if and only if $N$ contains $\rist(w)'$ for some $n\ge 1$ and all $w\in T^{(n)}_v$. 
\end{lemma}

\begin{proof}
First we observe that $\partial T_v$ can be partitioned into finitely many minimal, clopen $\rist(v)$-invariant subsets, since $\rist(v)$ has finite index in the image of the action of $G_v$ on $\partial T_v$, which acts minimally on $\partial T_v$. Now let $N\unlhd \rist(v)$ and consider the set $C\subset \partial T_v$ of $N$-fixed points. It is closed and $\rist(v)$-invariant, hence both $C$ and $\partial T_v\setminus C$ are clopen. 

Every $\xi\in \partial T_v\setminus C$ admits a prefix $w$ such that $f(w)\neq w$ for some $f\in N$. It follows from Lemma \ref{l-double-comm} that there is a set $Q\subset T_v$  such that $\partial T_v\setminus C=\cup_{w\in Q} \partial T_w$ and $\rist(w)'\le N$ for $w\in Q$. By compactness we can extract a finite cover and then refine it so that $ Q\subset T^{(n)}_v$ for some $n$. Then for every $w\in T^{(n)}_v\setminus Q$, we have $\partial T_w\subset C$, so $N$ fixes $w$ and acts trivially on $T_w$.

If $Q=T^{(n)}_v$, then $N$ contains $\rist(w)'$ for every $w\in T^{(n)}_v$ and $\rist(v)/N$ is virtually abelian. Else, $\rist(w)$ maps injectively to $\rist(v)/N$ for every $w\notin Q$, and thus $\rist(v)/N$ cannot be virtually abelian: indeed every normal subgroup of $\rist(w)$ contains $\rist(u)'$ for some $u$ by Lemma \ref{l-double-comm}, and it is easy to see that $\rist(u)'$ is never abelian---a fortiori $\rist(w)$ is not virtually abelian.  \qedhere
    
\end{proof}

\subsection{Commensurating actions of branch groups} \label{s-commensurating-branch}

\begin{thm} \label{t-multi-ended-graphs}
Let $G\le \aut(T)$ be a finitely generated branch group whose action on $\partial T$ has finite groups of germs.
Let $H\le G$ be a subgroup such that the action of $G$ on $G/H$ is faithful and $\esf(\Gamma_{G/H})\ge 2$. Then there is $\xi \in \partial T$ such that $H$ is a finite index subgroup of $G_\xi$. Furthermore, we have $\esf(\Gammatilde_\xi)\ge \esf(\Gamma_{G/H})$, and for every commensurated subset $A\subset \Gamma_{G/H}$, there is a commensurated subset $\widetilde{A}\subset \Gammatilde_
\xi$ such that $\nu_A-\alpha\nu_{\widetilde{A}}\in \ell^\infty(G)$, where $\alpha=\frac{[G^0_\xi\,  :\,  G^0_\xi\cap H]}{[H\, : \, G^0_\xi\cap H]}$.

\end{thm}

\begin{proof}
Fix a finite symmetric generating set $S$ of $G$, and let  $H\le G$ be a subgroup as in the statement.

Given a geodesic ray $\xi\in \partial T$ in a rooted tree $T$, we write $N(\xi)\subset T$ for the collection of all vertices that are not on $\xi$, but whose parent is in $\xi$.   Note that we have	\[\partial T\setminus \{\xi\}=\bigsqcup_{v\in N(\xi)} \partial T_v.\] 
The starting point of the proof is the following application of Theorem \ref{t-commensurating-product-discrete} (which does not require the assumption that groups of germs are finite).

\begin{lemma}[First key step] \label{l-commensurating-branch}
There is $\xi\in \partial T$ and a finite index subgroup $K_v \le \rist(v)$ for each  $v\in  N(\xi)$ such that 
	\begin{equation}\label{e-subgroup} \bigoplus_{v\in N(\xi)} K_v \le H \le G_{\xi}.\end{equation}\end{lemma}
Here and throughout, we identify $\bigoplus_{v\in N(\xi)} K_v$ with the subgroup $\langle K_v, v\in N(\xi)\rangle$.
\begin{proof}[Proof of Lemma \ref{l-commensurating-branch}]
Let us set $X=G/H$, and let $A\subset X$ be a commensurated subset which is not transfixed. Fix $x\in X$, and let $n\ge 1$. Obviously the action of $\rist(n)$ on $X$ must have infinite image, since otherwise the action of $G$ would have finite image and thus $X$ would be finite. By Lemma \ref{l-reduction-transitive}, the action of $\rist(n)\simeq \prod_{v\in T^{(n)}}\rist(v)$ on $X$ has finitely many orbits, and for each such orbit $Y$ the sets $Y\cap A$ and $Y\cap A^c$ are both infinite.  Let $Y_n$ be the $\rist(n)$-orbit of $x$; it is infinite by Lemma \ref{l-reduction-transitive}. Hence there is at least one $v\in T^{(n)}$ such that $\rist(v)\acts Y_n$ has infinite image. Suppose, by contradiction, that there are two such distinct $v, w\in T^{(n)}$. Then we can write $\rist(n)=G_1\times G_2$, where $G_i\acts Y_n$ has infinite image for $i=1, 2$. Then  Theorem \ref{t-commensurating-product-discrete} implies that the action $\rist(n)\acts Y_n$ has virtually abelian image, and Lemma \ref{l-quotient-rist} implies that there is $m\ge n$ such that $\rist(m)'$ fixes every point of $Y_n$. But since $\rist(m)'$ is normal in $G$, and $G$ acts transitively on $X$, this implies that $\rist(m)'$ actually fixes every point of $X$, contradicting that $G$ acts faithfully on $X$.  Therefore, for each $n$, there exists a unique $v_n\in T^{(n)}$ such that the action of $\rist(v_n)$ on $Y_n$ has infinite image. The sequence $(v_n)$ forms a ray $\xi_x\in \partial T$. The map $x\mapsto \xi_x$ is equivariant, so that $G_x\le G_{\xi_x}$. In addition, for every $v\in N(\xi)$ at level $n$, we have $v\neq v_n$, so there is a finite index subgroup $K_v\le \rist(v)$ that acts trivially on $Y_{n}$, and thus fixes $x$. It follows that $\bigoplus_{v\in N(\xi_x)} K_v\le G_x\le G_{\xi_x}$. Choosing $x=eH$ we get the desired conclusion. \qedhere
\end{proof}

From now on, we let $\xi=(v_n)\in \partial T$ and $(K_v)_{v\in N(\xi)}$ be given by Lemma \ref{l-commensurating-branch}. We set $K=\bigoplus_{v\in N(\xi)} K_v$, so that \eqref{e-subgroup} reads $K\le H\le G_\xi$. In the reminder of the proof we assume that the group of germs $G_\xi /G^0_\xi$ is finite, and progressively strengthen the conclusion of Lemma \ref{l-commensurating-branch} until reaching the conclusion of the theorem.

For each $v\in N(\xi)$, denote the set of maps $\{g|_{T_v}\colon T_v\to T_{g(v)} \,  \vert \, g\in G\}$ by $E_v$.
The group $K_v$ acts on $E_v$ from the right by precomposition, and the condition that $K_v$ has finite index in $\rist(v)$ implies that the quotient $E_v/K_v$ is a finite set. Then for every $g\in G$, the data 
\[(g|_{T_v}\text{ mod }K_v)_{v\in N(\xi)}
\in \prod_{v\in N(\xi)} E_v/K_v\]
determines uniquely the coset $g K$, and thus the coset $gH$. We will repeatedly use this observation. 
 
Since $G^0_\xi$ has finite index in $G_\xi\ge H$, the group $H\cap G^0_\xi$ has finite index in $H$. Moreover Lemma \ref{lem:endsofcoverings} ensures that $\esf(\Gamma_{G/(H\cap G^0_\xi)})\ge \esf(\Gamma_{G/H})$  and that for every commensurated subset $A\subset{G/H}$, there is a commensurated subset $B\subset {G/(H\cap G^0_\xi)}$  such that $\nu_A=[H\colon H\cap G^0_\xi]^{-1}\nu_B$.  Hence, to prove the theorem, there is no loss of generality in replacing $H$ by $H \cap G^0_\xi$, i.e. we can suppose that $H\le G^0_\xi$ to begin with. After this preliminary reduction,  the map $gH\mapsto gG^0_\xi$ induces a covering map 

\[p\colon \Gamma_{G/H}\to  \widetilde{\Gamma}_\xi.\]

\begin{lemma}\label{l-fiber-disconnected}
Every connected infinite subset $U\subset \Gamma_{G/H}$ has an infinite projection $p(U)$.
\end{lemma}
\begin{proof}
Suppose that $U$ is connected and $p(U)$ is finite, and let us show that $U$ is finite. Without loss of generality, we may assume that $U$ contains the basepoint $eH$ of $\Gamma_{G/H}$ (upon enlarging $U$ by adding to it a path from the basepoint to any point in $U$). Let $R>0$ be such that $p(U)\subset B_{\Gammatilde_\xi}(eG^0_\xi, R)$. By Proposition~\ref{pr:convergenceofgraphs}, for all $m$ large enough the map $g\cdot v\mapsto [g, \xi]$ is an isomorphism $B_{\Xi_m}(G^0_v, R)\arr B_{\Gammatilde_\xi}(G^0_\xi, R)$, where $v$ is the prefix of length $m$ of $\xi$.

Suppose that $s_1, \ldots, s_n\in S$ is any sequence of generators such that the path $\gamma=(eH, s_1H, s_2s_1H, \ldots, s_n\cdots s_1H)$ is contained in $U$, and thus its projection $p(\gamma)$ is contained in $B_{\Gammatilde_\xi}(G^0_\xi, R)$. 
Then the preimage of $p(\gamma)$ in $B_{\Xi_m}(eG^0_v, R)$ is the path 
$G^0_v, s_1G^0_v, s_2s_1G^0_v , \ldots, s_n\cdots s_1G^0_v$. Since it stays inside $B_{\Xi_m}(G^0_v, R)$, the maps $s_i\cdots s_2s_1|_{T_v}$ belong to a finite set.  The value of $s_n\cdots s_1\vert_{T_v}$ determines uniquely the map $s_n\cdots s_1|_{T_w}$ for all $w\in T_v$, and hence for all but finitely many $w\in N(\xi)$. We deduce that  $(s_n\cdots s_1|_{T_w}\text{ mod }K_w)_{w\in N(\xi)}$ can take only finitely many values. Since the latter determines uniquely the endpoint of $\gamma$, and every point of $U$ can be connected to the trivial coset by a path inside $U$, this shows that $|U|<\infty$.  \qedhere
\end{proof}

Recall that given a graph $\Gamma$ and a subset $F$ of vertices, we denote by $\pi_0(\Gamma\setminus F)$ the set of connected components of $\Gamma\setminus F$, and by $\pi_0(\Gamma\setminus F)_\infty$ its subset consisting of infinite connected components. Fix a finite subset $\Sigma\subset\Gamma_{G/H}$ such that $|\pi_0(\Gamma_{G/H}\setminus\Sigma)_\infty|\ge 2$.  For $U\in \pi_0(\Gamma_{G/H}\setminus\Sigma)_\infty$, we set $p'(U):=p(U)\setminus p(\Sigma)$.

\begin{lemma}\label{l-projection-union-cc}

For every $U\in \pi_0(\Gamma_{G/H}\setminus\Sigma)_\infty$, $p'(U)$ is non-empty and it is a union of elements of $\pi_0(\Gammatilde_\xi\setminus p(\Sigma))$.
    
\end{lemma}
\begin{proof}
    The fact that $p'(U)$ is non-empty follows from Lemma \ref{l-fiber-disconnected}, since $p(\Sigma)$ is finite. Let $x$ be any point in $p'(U)$, and $Q\in \pi_0(\Gammatilde_\xi\setminus p(\Sigma))$ be its connected component. For any point $y\in Q$, we can find a path $\gamma\subset Q$ starting at $x$ and ending at $y$. The lift of $\gamma$ starting at any point in $p^{-1}(x)\cap U$ avoids $\Sigma$ (since it projects to $\gamma$ which avoids $p(\Sigma)$), hence it is entirely contained in $U$. Since its endpoint projects to $y$, it follows that $y\in p'(U)$. Since $y\in Q$ is arbitrary this shows that $Q\subset p'(U)$. We have shown that $p'(U)$ contains any component  $Q\in \pi_0(\Gammatilde_\xi\setminus p(\Sigma))$ that it intersects non-trivially, and so it is a union of such components. \qedhere

\end{proof}

\begin{lemma}[Second key step] \label{l-ends-injection}
    Let $U_1, U_2\in \pi_0(\Gamma_{G/H}\setminus\Sigma)_\infty$ be such that $p'(U_1)\cap p'(U_2)$ contains some infinite component $Q\in \pi_0(\Gammatilde_\xi\setminus p(\Sigma))_\infty$. Then $U_1=U_2$.
\end{lemma}
\begin{proof}
    Fix $xH\in U_1, yH\in U_2$ such that $p(xH)=p(yH)\in Q$. Then $xG^0_\xi=yG^0_\xi$ and so we can rewrite $y=xk$ where $k=x^{-1}y\in G^0_\xi$. We fix such $x$ and $k$ for the rest of the proof. Since $k\in G^0_\xi$, we can choose a prefix $w$ of $\xi$ such that $k\in G_w^0$. Consider the subgroup
\[R=\langle K_v \colon v\in N(\xi) \colon |v|\le |w|\rangle\le G_w^0.\]
(where the $K_v\leq \rist(v)$ are still the subgroups from \eqref{e-subgroup}). Then $R$ has finite index in $G_w^0$.  Recall that in a finitely generated branch group $G$, all rigid stabilizers $\rist(v)$ are finitely generated, and hence so are the groups $K_v$, hence the group $R$ is finitely generated. As a consequence, the group  $G_w^0$ is also finitely generated. Therefore, there is a finite index subgroup $R_0<R$ which is characteristic in $G_w^0$, and thus normal in the whole vertex stabilizer $G_w$.

Let $\Omega\subset G_w^0$ be a set of representatives of the classes of the finite quotient $G_w^0/R_0$. Set $C=\sup\{\ell_S(t), t\in \Omega\}$, where $\ell_S(\cdot)$ denotes the word length.

We have $G^0_\xi\le G_w$, and the set of cosets $G_w/G^0_\xi$ defines a subset of the set of vertices  of the graph of germs $\Gammatilde_\xi$. Since $G_w$ has finite index in $G$, there exists some $D>0$ such that every vertex of $\Gammatilde_\xi$ is at distance at most $D$ from $G_w/G^0_\xi$. On the other hand, since $Q$ is a connected component of $\Gammatilde_\xi \setminus p(\Sigma)$, the ball of radius $D$ around all but finitely many points of $Q$ is contained in $Q$. In particular, $Q$ has infinite intersection with $G_w/G^0_\xi$.  Therefore we can find a path $\gamma$ in $Q$ that starts at $xG_\xi^0=p(xH)=p(xkH)$ and ends at  a point  of $G_w/G^0_\xi$ at distance $>C$ from $p(\Sigma)$. Let $h=s_n\cdots s_1$ be the word read along this path, so that its endpoint is $zG_\xi^0$, where $z=hx\in G_w$. Then the two lifts of $\gamma$ starting at $xH$ and $xkH$ both avoid $\Sigma$, and thus are contained in $U_1$ and $U_2$, respectively. In particular $zH\in U_1$ and $zkH\in U_2$.  Furthermore both points $zH$ and $zkH$ are at distance $>C$ from $\Sigma$. Let $t\in \Omega$ be a representative of $zk^{-1}z^{-1}$ in $G_w^0/R_0$. Then there is a path in $\Gamma_{G/H}$ of length at most $C$ going from $zkH$ to $tzkH$. This path avoids $\Sigma$, which implies that $tzkH\in U_2$. We have $zkz^{-1}t\in R_0$, by the choice of $t$. Since $k, z\in G_w$, and $R_0$ is normal in $G_w$, this implies that $z^{-1}tzk\in R_0\le H$, so
\[tzkH=z\cdot z^{-1}tzkH=zH\in U_1.\]
We have shown that $tzkH\in U_1\cap U_2\neq \varnothing$,
which implies that $U_1=U_2$. \qedhere
    
\end{proof}

Lemma \ref{l-ends-injection} immediately implies that $|\pi_0(\Gammatilde_\xi\setminus p(\Sigma))_\infty|\ge |\pi_0(\Gamma_{G/H}\setminus \Sigma)_\infty|$,
hence $\esf(\Gammatilde_\xi)\ge\esf(\Gamma_{G/H})\ge 2$. Let us argue that it also implies that $H$ has finite index in $G^0_\xi$ (and hence in $G_\xi$). Suppose by contradiction that $p\colon \Gamma_{G/H}\to \Gammatilde_\xi$ has infinite degree. Let $\gamma$ be a finite path in $\Gammatilde_\xi$ that visits all components in $\pi_0(\Gammatilde_\xi\setminus p(\Sigma))_\infty$. If $p$ has infinite degree, the path $\gamma$ admits infinitely many lifts to $\Gamma_{G/H}$. Hence it admits a lift $\tilde{\gamma}$ which avoids $\Sigma$ as well as all the (finitely many) finite components in $\pi_0(\Gamma_{G/H}\setminus \Sigma)$.  Therefore $\tilde{\gamma}$ is contained in a single infinite component $U\in \pi_0(\Gamma_{G/H}\setminus \Sigma)_\infty$, which then has the property that $p'(U)$ contains all components in $\pi_0(\Gammatilde_\xi\setminus p(\Sigma))_\infty$. Therefore Lemma \ref{l-ends-injection} implies that $U$ is the unique infinite component of $\Gamma_{G/H}\setminus \Sigma$, which contradicts our choice of $\Sigma$.

Finally let $\nu_A$ be an unbounded cardinal-definite function associated to a commensurated subset $A\subset G/H$. Upon changing $A$ in its commensurability class (which changes $\nu_A$ by a bounded amount) we can find a finite subset $\Sigma\subset \Gamma_{G/H}$ such that $A$ is a union of elements of $\pi_0(\Gamma_{G/H}\setminus \Sigma)_
\infty$. Then Lemma \ref{l-ends-injection} implies that  $|A\triangle p^{-1}(p'(A))|<\infty$, thus $\nu_A-\nu_{p^{-1}(p'(A))}$ is bounded. But $\nu_{p^{-1}(p'(A))}=[G^0_\xi \colon H]\nu_{p'(A)}$ by Lemma \ref{lem:endsofcoverings}, finishing the proof. \qedhere

\end{proof}

\begin{cor}\label{c-FW}
Let $G\le \aut(T)$ be a finitely generated branch group whose action on $\partial T$ has finite groups of germs. 
Then $G$ has Property FW if and only if it is just-infinite and all the graphs of germs of its action on $\partial T$ are one-ended. 
\end{cor}
\begin{proof}
    If $G$ is not just-infinite, then by Proposition \ref{p-Grigorchuk} it admits an infinite virtually abelian quotient, and thus it does not have Property FW (in fact, every infinite virtually abelian group has Property PW \cite[Proposition 5.C.2]{Cor-FWsurvey}). If $G$ is just-infinite, then all its infinite Schreier graphs are faithful, and the conclusion follows from Theorem \ref{t-multi-ended-graphs}. 
\end{proof}
\begin{proof}[Proof of Theorem \ref{t-intro-germs}]
    Combine Theorem \ref{t-multi-ended-graphs} and Proposition \ref{p-Grigorchuk}. The last statement is Corollary \ref{c-FW}. \qedhere\end{proof}
 A \emph{virtual homomorphism} $\phi\colon G\dashrightarrow \Z$ is a homomorphism defined on a finite-index subgroup $\mathrm{Dom}(\phi)$. Such a homomorphism can be induced to a function $\widehat{\phi}\colon G\to \Z$ by choosing a system $\Omega$ of coset representatives of $\mathrm{Dom}(\phi)$, and setting $\widehat{\phi}(ht)=\phi(h)$ for $h\in \mathrm{Dom}(\phi)$ and $t\in \Omega$. Two choices of $\Omega$ yield extensions at bounded distance. So we can abuse terminology, and say that a function on $G$ is at bounded distance from a virtual homomorphism  $\phi\colon G\dashrightarrow \Z$ if it is at bounded distance from $\widehat{\phi}$. With this terminology, \cite[Proposition 6.B.3]{Cor-FWsurvey} states that any cardinal-definite function on a virtually abelian group $H$ is at bounded distance from $\sum_{i=1}^r|\phi_i|$, where $\phi_i\colon H\dashrightarrow \Z$ are  virtual homomorphisms. Hence Theorem \ref{t-multi-ended-graphs}, combined with Proposition \ref{p-Grigorchuk}, yields the following. 

\begin{cor} \label{c-cardinal-definite}
    For $G\le \aut(T)$ as in Theorem \ref{t-multi-ended-graphs}, every cardinal-definite function on $G$ is at bounded distance from  a function of the form \[\sum_{i=1}^r \alpha_i\nu_{A_i} +\sum_{i=1}^s|\phi_i|,\] where
each $A_i$ is a commensurated subset of $\Gammatilde_{\xi_i}$ for some $\xi_i\in \partial T$, each  $\alpha_i\in \mathbb{Q}_{>0}$, and the $\phi_i\colon G\dashrightarrow \Z$ are finitely many virtual homomorphism.  
    
\end{cor}

\section{Contracting self-similar groups}
\label{s-cutpoints}
    

    \subsection{Main definitions}
\label{s-preliminaries-contracting}
Let $\alb$ be a finite alphabet. The set $\alb^\ast$ of finite words in $\alb$ is naturally a rooted tree, with root the empty word $\varnothing$, where each $v\in \alb^\ast$ is connected by an edge to $vx$ for every $x\in \alb$. Hence the subtree rooted at $v$ is the set $v\alb^\ast$ of words beginning with $v$. The boundary $\partial \xs$ of the tree is identified with the set $\xo$ of right-infinite sequences $x_1x_2\cdots$.

		For every $g\in \aut(\alb^\ast)$ and every $v\in \alb^\ast$, there exists a unique element $g|_v\in \aut(\alb^\ast)$, called the \emph{section} of $g$ at $v$, satisfying 	\[g(vw)=g(v)g|_v(w),\]
		for every $w\in \alb^\ast$. Sections satisfy the cocycle rule $gh|_v=g|_{h(v)}h|_v$.
        
		\begin{defin}
        \label{def:selfsimilar1}
		A subgroup $G\le \aut(\alb^\ast)$ is \emph{self-similar} if for every $g\in G$ and $v\in \alb^\ast$, we have $g|_v\in G$. 
		\end{defin} 

        For a self-similar subgroup $G\le\aut(\alb^\ast)$, we denote by $\bim$ the set $\alb\times G$ together with commuting left and right actions given by
        \[g_1\cdot (x\cdot h)\cdot g_2=g_1(x)\cdot g_1|_xhg_2.\]
        Note that the right action is free and has orbits $x\cdot G$ for $x\in\alb$.

        One can interpret $\bim$ as the set of maps $x\cdot g\colon w\mapsto xg(w)\colon\xs\arr\xs$. Then the left and the right actions are post- and pre-compositions with the action of $G$ on $\alb^\ast$.

More generally, we adopt the following definition.

\begin{defin}
    A \emph{self-similar group} is a pair $(G, \bim)$ consisting of a group $G$ and a \emph{$G$-biset} $\bim$, i.e., a set with commuting left and right actions of $G$ on it, such that the right action is free and has finitely many orbits. A subset $\alb\subset\bim$ is a \emph{basis} of the biset if it intersects every right orbit once.
\end{defin}

If $\bim_1, \bim_2$ are a right $G$-set and a left $G$-set, respectively, then we denote by $\bim_1\otimes\bim_2$ the set of orbits of the action $g\cdot (x, y)=(x\cdot g^{-1}, g\cdot y)$ of $G$ on $\bim_1\times\bim_2$. We write the orbit of $(x_1, x_2)\in\bim_1\times\bim_2$ as $x_1\otimes x_2$, so that we have $x_1\cdot g\otimes x_2=x_1\otimes g\cdot x_2$ for all $x_i\in\bim_i$ and $g\in G$. 

If $\bim_1, \bim_2$ are $G$-bisets, then the set $\bim_1\otimes\bim_2$ is a biset with respect to the actions
\[g_1\cdot (x_1\otimes x_2)\cdot g_2=(g_1\cdot x_1)\otimes (x_2\cdot g_2).\]
It is easy to see that the bisets $\bim_1\otimes(\bim_2\otimes\bim_3)$ and $(\bim_1\otimes\bim_2)\otimes\bim_3$ are isomorphic, so we can define, for every biset $\bim$, the biset $\bim^{\otimes n}$, for every $n\ge 1$. We also set $\bim^{\otimes 0}=G$ with the natural left and right actions by multiplication.

For every self-similar group $(G, \bim)$ and for every $n\ge 0$, we have the natural action of $G$ on the set of right orbits $\bim^{\otimes n}/G$. We say that $(G, \bim)$ is \emph{faithful} if the left action of $G$ on $\bigsqcup_{n\ge 0}\bim^{\otimes n}/G$ is faithful. Otherwise, if $K$ is the kernel of the action, then the self-similar group $(G/K, \bim/K)$ is called the \emph{faithful quotient} of $(G, \bim)$.

If $\alb\subset\bim$ is a basis, then $\alb^{\otimes n}$ is a basis of the biset $\bim^{\otimes n}$, hence the set of right orbits $\bim^{\otimes n}/G$ is naturally identified with $\alb^{\otimes n}$. 

We write $x_1\otimes x_2\otimes\cdots\otimes x_n\in\alb^{\otimes n}$ as $x_1x_2\ldots x_n$, and $\alb^{\otimes n}=\alb^n$. After identification of $\bim^{\otimes n}/G$ with $\alb^n$, the left action of $G$ on $\bim^{\otimes n}/G$ is identified with an action of $G$ on $\alb^n$. The constructed action of $G$ on $\xs=\bigcup_{n\ge 0}\alb^n$ is self-similar in the sense of Definition~\ref{def:selfsimilar1}. We denote the image of $v\in\xs$ under $g$ for this action by $g(v)$, while the image of $v\in\bim^{\otimes n}$ with respect to the left action in the biset $\bim^{\otimes n}$ is denoted by $g\cdot v$. For every $g\in G$ and $v\in \alb^n\subset \bim^{\otimes n}$ there is a unique element of $G$ denoted $g|_v$ such that 
\[g\cdot v=g(v)\cdot g|_v\]
in $\bim^{\otimes n}$. The action of $G$ on $\xs$ is defined by the recurrent formula
\[g(vw)=g(v)g|_v(w).\]
for $v, w\in\xs$.

If $S$ is a generating set of $G$, and $\alb$ is a basis of $\bim$, then the self-similar group $(G, \bim)$ is uniquely defined by the equalities
\[g\cdot x=y\cdot g|_x,\]
for $g\in S$, $x, y=g(x)\in\alb$, where $g|_x\in G$ is written as a product of $S\cup S^{-1}$. 

\begin{defin}
    A self-similar group $(G, \bim)$ is said to be \emph{self-replicating} if the left action of $G$ on $\bim$ is transitive. Equivalently, if the action of $G$ on the first level $\alb$ of the tree $\xs$ is transitive and the map $g\mapsto g|_x\colon G_x\longrightarrow G$ from the stabilizer of $x\in\alb$ to $G$ is onto (for some and hence for all $x\in\alb$).
\end{defin}

If $(G, \bim)$ is self-replicating, then the left action of $G$ on $\bim^{\otimes n}$ is transitive for every $n$.

For $w\in\bim^{\otimes n}$, we denote by $G_w$ the stabilizer of the right orbit $w\cdot G$ with respect to the left action of $G$ on $\bim^{\otimes n}$. In other words, $G_w$ for $w\in\alb^n$ is the stabilizer of $w$ for the action of $G$ on $\xs$, coherently with the notation in \S\ref{s-preliminaries-trees} for arbitrary groups acting on rooted trees. We denote by $G_w^0$ the stabilizer of $w\in\bim^{\otimes n}$ with respect to the left $G$-action. If $w\in\alb^n$, then $G_w^0=\{g\in G\colon g(w)=w, g|_w=1\}$ again coherently with the notation in \S\ref{s-preliminaries-trees}.

Note that it then follows that the graph $\Xi_n$ defined in Definition \ref{d-graphs-trees} coincides with the Schreier graph $\Gamma_{\bim^{\otimes n}}$ of the left action of $G$ on $\bim^{\otimes n}$ (which is transitive if $G$ is self-replicating). We shall always identify the vertex set of $\Xi_n$ with $\bim^{\otimes n}$ in this case.

\begin{defin} \label{def:contracting}
    A self-similar group $(G, \bim)$ with basis $\alb$ is \emph{contracting} if there exists a finite set $\nuke\subset G$ such that for every $g\in G$, there exists $m$ such that $g|_v\in \nuke$ for all $v\in \alb^*,|v|\ge m$. The smallest set $\nuke$ with this property is called the \emph{nucleus} of $G$. 
\end{defin}

One can show (see~\cite[Corollary 2.11.7]{Nek:book}) that the property of a self-similar group to be contracting does not depend on the choice of the basis $\alb$. However, the nucleus depends on the choice of $\alb$.

In this paper, we will use several times the fact that the action of a faithful contracting self-similar group on the tree boundary $\xo$ has finite groups of germs \cite[Proposition 4.1]{Nek:free}.

We will also need the following result (see~\cite[Lemma~3.5.2]{Nek:book}).
\begin{lemma}
\label{lem:connectivityB}
Suppose that $(G, \bim)$ is self-replicating, contracting, and finitely generated, and let $\alb\subset\bim$, $\nuke\subset G$ be a basis and the associated nucleus. There exists a finite set $L\subset G$ such that for any $n\ge 1$ and any $v, u\in\alb^n$ there exists a sequence $h_1, h_2, \ldots, h_m\in\nuke$ such that
\[h_mh_{m-1}\cdots h_1(v)=u, \quad h_mh_{m-1}\cdots h_1|_v=e\]
 and $(h_kh_{k-1}\cdots h_1)|_v\in L$ for all $1\le k\le m$.

\end{lemma}

Note that the condition $h_mh_{m-1}\cdots h_1(v)=u, \quad h_mh_{m-1}\cdots h_1|_v=e$ can be rephrased as the equality $h_mh_{m-1}\cdots h_1 \cdot v=u$
in the set $\bim^{\otimes n}$. In other words, Lemma \ref{lem:connectivityB} says that any two points $v, u\in\alb^n\subset \bim^{\otimes n}$ can be connected by a path in the graph $\Xi_n$ (identified with the Schreier graph $\Gamma_{\bim^{\otimes n}}$) contained inside the set $\alb^n\cdot L$. Lemma~\ref{lem:connectivityB} implies the following.

\begin{lemma}
    \label{lem:extendingpaths}
Let $L$ be a subset as in Lemma \ref{lem:connectivityB}. Consider a path $v_1, v_2, \ldots, v_k$ in $\Xi_n$. Then for arbitrary $w_1, w_2\in\alb^m$ there is a path from $w_1v_1$ to $w_2v_k$ in $\Xi_{n+m}$ contained in $\bigcup_{i=1}^k \alb^m\cdot L\cdot v_i$. 
\end{lemma}

\begin{proof}
    Let $s_i\in\nuke$ be such that $s_i\cdot v_i=v_{i+1}$. There exist $u_i\in\alb^m$ and $r_i\in\nuke$ such that $r_i|_{u_i}=s_i$ (otherwise, $\nuke$ is not the minimal set satisfying the conditions of Definition~\ref{def:contracting}). Then $u_iv_i$ is adjacent to $r_i\cdot u_iv_i=r_i(u_i)v_{i+1}$ in $\Xi_{n+m}$. By Lemma~\ref{lem:connectivityB}, we can connect $r_i(u_i)$ to $u_{i+1}$ by a path in $\Xi_m$ staying inside $\alb^m\cdot L$. Appending to it $v_i$ and $v_{i+1}$ on the right we get a path from $r_i\cdot u_iv_i$ to $u_{i+1}v_{i+1}$ in $\Xi_{n+m}$ staying inside $\alb^m\cdot L\cdot v_{i+1}$. The union of these paths together with the edges from $u_iv_i$ to $r_i\cdot u_iv_i$ will be a path from $u_1v_1$ to $u_kv_k$ inside $\bigcup_{i=1}^k \alb^m\cdot L\cdot v_i$. Using again Lemma \ref{lem:connectivityB}, we can concatenate it with a path from $w_1v_1$ to $u_1v_1$ and another path from $u_kv_k$ to $w_kv_k$ staying inside the same set, to obtain a path satisfying the desired conclusion.
\end{proof}

The class of contracting self-similar group is distinct from the class of branch groups---and neither is included in the other. However the two classes have a large intersection, and it is on that intersection that our results become the most satisfactory. 
 We say that $(G, \bim)$ is a \emph{self-similar branch group} if it is faithful and branch with respect to its embedding in $\aut(\xs)$ for some basis $\alb$. We also recall the following terminology.
 
 \begin{defin}
     A faithful self-similar group $(G, \bim)$ is \emph{regular branch} over a finite index subgroup $K\le G$ if $K^{\alb}\le K$ for some basis $\alb$. Here we identify $K^{\alb}$ with the subgroup of $\aut(\xs)$ of elements that fix the first level and have sections in $K$.
 \end{defin} If $G$ is regular branch over $K$, it follows that every rigid stabilizer $\rist(v)$ contains a canonical copy of $K$.
\subsection{Limit spaces}
Fix a contracting self-similar group $(G, \bim)$, let $\alb\subset\bim$ be a basis, and let $\nuke$ be the associated nucleus. Consider the space $\alb^{-\omega}\times G$ of sequences $\ldots x_2x_1\cdot g$, where $G$ is discrete. We say that two sequences $\ldots x_2x_1\cdot g$ and $\ldots y_2y_1\cdot h$ are \emph{asymptotically equivalent} if there exists a finite set $N\subset G$ and a sequence $g_k\in N$ such that $g_k\cdot x_k\ldots x_2x_1\cdot g=y_k\ldots y_2y_1\cdot h$ in $\bim^{\otimes k}$ for every $k$. The quotient of $\alb^{-\omega}\times G$ by this equivalence relation is denoted by $\limg$ and is called the \emph{limit $G$-space}. The natural right action of $G$ on $\alb^{-\omega}\times G$ preserves the equivalence relation, hence it induces a right action of $G$ on $\limg$.

It is shown in~\cite[Proposition~3.2.6]{Nek:book} that $\ldots x_2x_1\cdot g$ and $\ldots y_2y_1\cdot h$ are asymptotically equivalent if and only if there is a sequence $s_n\in\nuke$, $n\ge 0$, such that $s_n\cdot x_n=y_n\cdot s_{n-1}$ for all $n\ge 1$, and $s_0g=h$.

It is shown in~\cite[\S~3.2]{Nek:book} that the $G$-space $\limg$ does not depend on the choice of the basis $\alb$ (up to topological conjugacy of $G$-spaces), that $\limg$ is locally compact and metrizable, and that the action of $G$ on $\limg$ is proper and (obviously) co-compact.

The following is proved in~\cite[Theorem~3.5.1]{Nek:book}.

\begin{prop}
\label{prop:locconnected}
Every finitely generated contracting self-replicating group is generated by its nucleus. The limit $G$-space $\limg$ is path-connected and locally path-connected if $(G, \bim)$ is self-replicating.
\end{prop}

For every $w=v\cdot h\in\bim^{\otimes n}$, where $v\in\alb^n$ and $h\in G$, the map 
\[\ldots x_2x_1\cdot g\mapsto \ldots x_2x_1g(v)\cdot g|_vh\]
preserves the asymptotic equivalence relation and hence induces a continuous map, which we denote as $\xi\mapsto \xi\otimes w$. The defined maps induce a $G$-equivariant homeomorphism $\limg\otimes\bim^{\otimes n}\arr\limg$, which we call \emph{natural}.
It is given, in terms of the encoding of $\limg$ as a quotient of $\xmo\times G$ by
\[(\ldots x_2x_1\cdot g)\otimes v\cdot h=\ldots x_2x_1g(v)\cdot g|_vh\]
for $\ldots x_2x_1\in\xmo$, $v\in\alb^n$, and $g, h\in G$.

It is shown in~\cite[Proposition~4.5.30]{nek:dyngroups} (see also~\cite[Theorem~2.11]{nek:IMGdimension}) that there exists a $G$-invariant metric $d$ on $\limg$ such that the maps $\xi\mapsto \xi\otimes w$ are uniformly contracting, i.e., such that there exist $C, L>1$ such that
\[d(\xi_1\otimes w, \xi_2\otimes w)\le CL^{-n}d(\xi_1, \xi_2)\]
for all $\xi_1, \xi_2\in\limg$ and $w\in\bim^{\otimes n}$ for all $n\ge 0$. The following characterization of the limit $G$-space $\limg$ is proved in~\cite[Theorem~3.4.13]{Nek:book}.

\begin{thm}
\label{th:uniquenessoflimg}
Let $(G, \bim)$ be a finitely generated contracting self-similar group. Let $\mathcal{X}$ be a locally compact metric space with a proper co-compact right action of $G$ on it by isometries. Suppose that there exists a $G$-equivariant homeomorphism $\Phi\colon\mathcal{X}\otimes\bim\arr\mathcal{X}$ such that the maps $\mathcal{X}\arr\mathcal{X}\colon\xi\mapsto \Phi(\xi\otimes x)$ are  uniform contractions for every $x\in\bim$. Then there is a $G$-equivariant homeomorphism $\mathcal{X}\arr\limg$ conjugating $\Phi$ with the natural homeomorphism $\limg\otimes\bim\arr\limg$.
\end{thm}

Consider a left action of $G$ on a topological space $M$. The space $\limg\otimes M$ is homeomorphic to the quotient of the space $\alb^{-\omega}\times M$ by the equivalence relation identifying $\ldots x_2x_1\cdot t_1$ with $\ldots y_2y_1\cdot t_2$ if there exists a sequence $s_n\in\nuke$ for $n\ge 0$, such that $s_n\cdot x_n=y_n\cdot s_{n-1}$, for all $n\ge 1$, and $s_0\cdot t_1=t_2$.

If the space $M$ is discrete and the action of $G$ on it is transitive, then $\limg\otimes M$ is homeomorphic to the quotient $\limg/G_x$ by the stabilizer of a point $x\in M$. The homeomorphism maps $\xi\otimes g\cdot x$ to the image of $\xi\cdot g$ in $\limg/G_x$. 

If we consider the trivial action of $G$ on a point, then we get the \emph{limit space} $\lims=\limg/G$, equal to the quotient of the space $\alb^{-\omega}$ of left-infinite sequences $\ldots x_{2}x_{1}$ by the asymptotic equivalence relation identifying two sequences $(x_{n})_{n\ge 1}$ and $(y_{n})_{n\ge 1}$ if there exists a sequence $s_n\in\nuke$ such that $s_n(x_{n}\cdots x_{1})= y_{n}\cdots y_{1}$  for all $n\ge 1$.

If we consider the action of $G$ on $\xo$ of right-infinite sequences over the alphabet $\alb$ (i.e., on the boundary of the tree $\xs$), then we get the \emph{limit solenoid} $\limg\otimes\xo$ equal to the quotient of the space $\alb^\Z=\alb^{-\omega}\times \alb^{\omega}$ of bi-infinite sequences $\ldots x_{-2}x_{-1}.x_0x_1\ldots$ by the asymptotic equivalence relation identifying two sequences $(x_n)_{n\in\Z}$ and $(y_n)_{n\in\Z}$ if there exists a sequence $s_n\in\nuke$ such that $s_n\cdot x_n=y_n\cdot s_{n-1}$ for all $n\in\Z$.

 We denote by $\sigma$ the 
shifts $\ldots x_2x_1\mapsto \ldots x_3x_2$ and $\ldots x_{-2}x_{-1}.x_0x_1\mapsto \ldots x_{-3}x_{-2}.x_{-1}x_0\ldots$ of $\xmo$ and $\alb^{\Z}$. They preserve the asymptotic equivalence relation, and thus descend to the quotient to a map
\[\mathsf{s}\colon \lims\to \lims\]
 and to a homeomorphism
\[\widehat{\mathsf{s}}\colon \limg\otimes\xo \to \limg\otimes\xo. \]
The map $\mathsf{s}$ is called the \emph{limit dynamical system}. It is shown in \cite[\S5.6]{Nek:book} that  the map $\widehat{\mathsf{s}}$ is  its \emph{natural extension}: namely the limit solenoid $\limg\otimes\xo$ is homeomorphic to the inverse limit of 
\[\lims \xleftarrow{\mathsf{s}}\lims\xleftarrow{\mathsf{s}}\cdots \]
and $\widehat{\mathsf{s}}$ is conjugate to the homeomorphism induced by $\mathsf{s}$ on the inverse limit.

For $\xi\in\alb^\omega$, denote $\leaf_\xi=\limg\otimes G(\xi)$, where $G(\xi)$ is the $G$-orbit of $\xi$ seen as a discrete set. We call it the \emph{leaf} of $\xi$. The inclusion $G(\xi)\subset\xo$ induces an injective continuous map from $\leaf_\xi$ to the limit solenoid. In particular, the asymptotic equivalence relation on $\alb^{-\omega}\times G(\xi)$ defining $\leaf_\xi$ is given by the same condition as for the limit solenoid. The space $\leaf_\xi$ is naturally homeomorphic to the space $\limg/G_\xi$.

If $(G, \bim)$ is self-replicating, then the path-connected components of the solenoid $\limg\otimes\xo$ are the images of the leaves $\leaf_\xi$ under the maps $\leaf_\xi\arr\limg\otimes\xo$, defined above. Each of the path-connected components is dense in the solenoid in the self-replicating case.

Recall that we denote by $[G, \xi]$ the set of germs $[g, \xi]$ of elements of $G$ at $\xi$, with the discrete topology. We denote by $\widetilde{\leaf}_\xi$ the space $\limg\otimes [G, \xi]$. It is identified with $\limg/G^0_\xi$. We call it the \emph{germ-leaf} of $\xi$.

\begin{prop}
\label{pr:leavesmap}
    Suppose that $G$ is self-replicating. Then the inverse of the shift $\ldots x_{-2}x_{-1}. x_0x_1\ldots\mapsto \ldots x_{-1}x_0. x_1x_2\ldots$ induces a homeomorphism $\leaf_{x_0x_1\ldots}\arr\leaf_{x_1x_2\ldots}$.

    Similarly, the map $\ldots x_{-2}x_{-1}.[g, x_0x_1\ldots]\mapsto \ldots x_{-1}g(x_0).[g|_{x_0}, x_1x_2\ldots]$ induces a homeomorphism $\leaftilde_{x_0x_1\ldots}\arr\leaftilde_{x_1x_2\ldots}$.
\end{prop}    

\begin{proof}
It follows from the transitivity of the left action of $G$ on $\bim^{\otimes n}$ that the $G$-orbit of a sequence $\xi=x_1x_2\ldots$ is equal to the set of sequences $y_1y_2\ldots$ such that $y_{n+1}y_{n+2}\ldots$ belongs to the $G$-orbit of $x_{n+1}x_{n+2}\ldots$. Similarly, the set of germs of elements of $G$ at $\xi$ is equal to the set of germs of the transformations
\[x_1x_2\ldots x_na_{n+1}a_{n+2}\ldots\mapsto y_1y_2\ldots y_n g(a_{n+1}a_{n+2}\ldots)\]
for $g\in G$. Consequently, the inverse of the shift induces a bijection between $\xmo\times G(x_1x_2\ldots)$ and $\xmo\times G(x_2x_3\ldots)$ and a bijection between $\xmo\times [G, x_1x_2\ldots]$ and $\xmo\times [G, x_2x_3\ldots]$. It is easy to see that the asymptotic equivalence relations are preserved by the shift and its inverse, which finishes the proof.\qedhere\end{proof}

\subsection{Tiles}

From now on, we assume that our contracting group $(G, \bim)$ is self-replicating and finitely generated. As above we denote by $\sigma$ the 
shifts $\ldots x_2x_1\mapsto \ldots x_3x_2$ and $\ldots x_{-2}x_{-1}.x_0x_1\mapsto \ldots x_{-3}x_{-2}.x_{-1}x_0\ldots$ of $\xmo$ and $\alb^{\Z}$ as well as the maps they induce on quotient spaces (in particular, the homeomorphisms between the leaves described in Proposition~\ref{pr:leavesmap}).

We denote by $\sigma^{-1}$ the shift $x_1x_2\ldots\mapsto x_2x_3\ldots$ on $\xo$. For $\xi\in\xo$, we denote by $\sigma(\xi)$ an arbitrary preimage of $\xi$ under the shift $\sigma^{-1}$, i.e., any sequence $x\xi$ for $x\in\alb$. Since we assume that the action is self-replicating, the graphs $\Gamma_{\sigma(\xi)}, \Gammatilde_{\sigma(\xi)}$ and the spaces $\leaf_{\sigma(\xi)}$, $\leaftilde_{\sigma(\xi)}$ will not depend on the particular choice of $\sigma(\xi)$.

Denote by $\til$ the image of the set $\alb^{-\omega}\cdot 1$ of sequences ending with the identity element of $G$ in $\limg$. Then $\limg=\bigcup_{g\in G}\til\cdot g$. The space $\til$ is naturally homeomorphic to the quotient of $\xmo$ by the equivalence relation identifying $\ldots x_2x_1$ with $\ldots y_2y_1$ if there exists a sequence $s_n\in\nuke$, $n\ge 0$, such that $s_n\cdot x_n=y_n\cdot s_{n-1}$ and $s_0=1$.

If $M$ is a left $G$-space, then the space $\limg\otimes M$ is equal to the union of the \emph{tiles} $\til\otimes t$ for $t\in M$. 

Recall that, since we assume that $G$ is self-replicating and finitely generated, the nucleus $\nuke$ is a generating set of $G$. 
The first paragraph of the next proposition is proved in the same way as~\cite[Proposition~3.3.5]{Nek:book}. The second paragraph is straightforward.

\begin{prop}
\label{prop:tiles}
Let $M$ be a discrete left $G$-set. Two tiles $\til\otimes x, \til\otimes y$, for $x, y\in M$, intersect if and only if there exists $g\in\nuke$ such that $y=g(x)$.

For any subset $A\subset M$ the subspace $\bigcup_{x\in A}\til\otimes x$ depends only on the sub-graph of the graph of the action of $G$ on $M$ (with respect to the generating set $\nuke$) spanned by $A$. Namely, it is the quotient of the space $\xmo\times A$ by the equivalence relation identifying $\ldots x_2x_1\otimes x$ with $\ldots y_2y_1\otimes y$, for $x, y\in A$, if and only if there exists a sequence $s_n\in\nuke$, $n\ge 0$, such that $s_n\cdot x_n=y_n\cdot s_{n-1}$ and $s_0(x)=y$.
\end{prop}

The natural homeomorphism of $\limg\otimes\bim^{\otimes n}\arr\limg$ induces a homeomorphism of \[\limg\otimes (\bim^{\otimes n}\otimes M)\cong (\limg\otimes\bim^{\otimes n})\otimes M\arr\limg\otimes M.\]

The images of the tiles $\til\otimes v\otimes t$ of $\limg\otimes\bim^{\otimes n}\otimes M$ under these homeomorphisms are called the \emph{$n$th level tiles} of $\limg\otimes M$. In particular, the tiles $\til\otimes t$ for $t\in M$ are tiles of the $0$th level.

The tile $\til\otimes v\otimes t$, for $v\in\xs$  and $t\in M$, is the image in $\limg\otimes M$ of the set $\xmo v\cdot t\subset\xmo\times M$. In particular, each tile $\til\otimes v\otimes t$ of level $n$ is equal to the union $\bigcup_{x\in\alb}\til\otimes xv\otimes t$ of tiles of level $n+1$.

The homeomorphisms $\sigma\colon\leaf_{\sigma^{-1}(\xi)}\arr\leaf_{\xi}$ and $\sigma\colon\leaftilde_{\sigma^{-1}(\xi)}\arr\leaftilde_\xi$ map tiles of the level $n$ homeomorphically to tiles of the level $n-1$. Therefore, we can define the tiles of negative levels as the images of the tiles of the zeroth level under iterations of the shift.

A tile of level $-n$ of $\leaf_{x_1x_2\ldots}$ is the set \[\bigcup_{a_1a_2\ldots a_n\in\alb^n}\til\otimes a_1a_2\ldots a_ng(x_{n+1}x_{n+2}\ldots)\] for a fixed $g\in G$. Similarly, a tile of level $-n$ of $\leaftilde_{x_1x_2\ldots}$ is the set \[\bigcup_{a_1a_2\ldots a_n\in\alb^n}\til\otimes \gamma_{a_1a_2\ldots a_n},\] where $\gamma_{a_1a_2\ldots a_n}$ is the germ at $x_1x_2\ldots$ of the homeomorphism \[x_1x_2\ldots x_n\xo\arr a_1a_2\ldots a_n\xo\colon x_1x_2\ldots x_nw\mapsto a_1a_2\ldots a_ng(w)\] for a fixed $g\in G$. (The facts that these germs belong to the set of germs of $G$ and that $a_1a_2\ldots a_nx_{n+1}x_{n+2}\ldots$ belongs to the $G$-orbit of $x_1x_2\ldots$ follow from the condition that $(G, \bim)$ is self-replicating.)

Recall from \S\ref{s-preliminaries-contracting} that we denote by $\Xi_n$ the graph of the left action of $G$ on $\bim^{\otimes n}$ with respect to the generating set $\nuke$. It is isomorphic to the Schreier graph $\Gamma_{G/G_w^0}$ for any $w\in\bim^{\otimes n}$. It follows from Proposition~\ref{prop:tiles} that $\Xi_n$ can be seen as the adjacency graph of the tiles of the $n$th level of $\limg$.

Recall also that we denote by $\Gamma_n$ the graph of the action of $G$ on the $n$th level $\alb^n$ of the tree $\xs$, i.e., the Schreier graph of $G$ modulo $G_w$ for any $w\in\bim^{\otimes n}$. It is interpreted as the adjacency graph of the $n$th level tiles of $\lims$.

The orbital graph $\Gamma_\xi$ and the graph of germs $\Gammatilde_\xi$ of a point $\xi\in\xo$ are the adjacency graphs of the $0$th level tiles of the leaves $\leaf_\xi$ and $\leaftilde_\xi$, respectively. The adjacency graphs of the tiles of the $n$th level of $\leaf_\xi$ and $\leaftilde_\xi$ are $\Gamma_{\sigma^n(\xi)}$ and $\Gammatilde_{\sigma^n(\xi)}$, respectively.

Since the diameters of the tiles of the $n$th level exponentially go to zero as $n\to\infty$, we get that, informally, $\limg$ and $\lims$ are scaling limits of the graphs $\Xi_n$ and $\Gamma_n$, respectively. 

Recall also from \S\ref{s-preliminaries-trees} that the inclusions $G_{x_1x_2\ldots}\le G_{x_1x_2\ldots x_n}$ and $G_{x_1x_2\ldots x_n}^0\le G_{x_1x_2\ldots}^0$ induce covering maps of labeled directed graphs $\Gamma_{x_1x_2\ldots}\arr\Gamma_n$ and $\Xi_n\arr\Gammatilde_{x_1x_2\ldots}$. Explicitly the map $\Gamma_{x_1x_2\ldots}\arr\Gamma_n$ acts on the vertices by $y_1y_2\ldots\mapsto y_1y_2\ldots y_n$. The map $\Xi_n\arr\Gammatilde_{x_1x_2\ldots}$ acts on the vertices by $y_1y_2\ldots y_n\cdot g\mapsto [h, x_1x_2\ldots]$, where $h\in G$ is any element such that $h(x_1x_2\ldots x_n)=y_1y_2\ldots y_n$ and $h|_{x_1x_2\ldots x_n}=g$. By Proposition \ref{pr:convergenceofgraphs}, the injectivity radius of these covering maps tends to $\infty$ with $n$.

These maps induce in turn continuous surjective maps $\leaf_{x_1x_2\ldots}\arr\lims$ and $\limg\arr\leaftilde_{x_1x_2\ldots}$ (that depend on $n$), given by the rules obtained by appending left-infinite sequences to the above formulas for the maps on the vertices of the graphs.

Let us denote by $\mathcal{U}_n(\xi)$ the union of the $n$th level tiles containing a point $\xi$. The following is proved in the same way as~\cite[Proposition~3.4.1]{Nek:book}.

\begin{prop}
\label{pr:basisofnbhd}
    Let $M$ be a discrete left $G$-space. Then for every $\xi\in\limg\otimes M$, the set $\{\mathcal{U}_n(\xi)\}_{n\ge 0}$ is a basis of neighborhoods of $\xi$ in $\limg\otimes M$.
\end{prop}

As a corollary of Propositions~\ref{prop:tiles},~\ref{pr:convergenceofgraphs}, and~\ref{pr:basisofnbhd}, we get the following.

\begin{prop}
    \label{pr:localhomeomorphism}
For every $\xi\in\xo$, every point of $\leaftilde_\xi$ has an open neighborhood homeomorphic to an open subset of $\limg$. Every point of $\limg$ has an open neighborhood homeomorphic to an open subset of $\leaftilde_\xi$ for some $\xi$.

The same statements are true for the spaces $\leaf_\xi$ and $\lims$.
\end{prop}

Let $(H, \bim)$ be a contracting self-similar group, and let $\alb$ be a basis of $\bim$. Let $K$ be the kernel of the action of $H$ on $\xs$. Let $\pi\colon H\arr G=H/K$ be the natural epimorphism. The map $x\mapsto x\cdot K$ from $\alb$ to $\bim/K$ is injective, and its image is a basis of the $G$-biset $\bim/K$. We will identify $x\in\bim$ with $x\cdot K\in\bim/K$ and write $x\cdot K$ simply as $x$. Then the map $\pi$ agrees with the actions of $H$ and $G$ on $\xs$.

We get a natural continuous map $\limg[H]\arr\limg$ given by $\ldots x_2x_1\cdot g\mapsto\ldots x_2x_1\cdot\pi(g)$. In fact, $\limg$ is naturally homeomorphic to $\limg[H]/K$, and the map $\limg[H]\arr\limg$ is the quotient map $\limg[H]\arr\limg[H]/K$.

\begin{prop}
\label{pr:limgH}
 Suppose that the image in $G=H/K$ of every non-trivial element of the nucleus of $H$ is non-trivial. Then the natural map $\limg[H]\arr\limg\cong\limg[H]/K$ is a covering.
\end{prop}

\begin{proof}
It is enough to show that $K$ acts freely on $\limg[H]$, since the action of $H$ on $\limg[H]$ is proper. Suppose that $h\in H$ fixes a point $\xi\in\til\subset\limg[H]$, represented by a sequence $\ldots x_2x_1\cdot 1$. Then $\ldots x_2x_1\cdot 1$ and $\ldots x_2x_1\cdot h$ are asymptotically equivalent. Therefore, there exists a sequence $h_n$ of elements of the nucleus $\nuke_H$ of $H$ such that $h_n\cdot x_n=x_n\cdot h_{n-1}$ and $h_0=h$. In particular, $h\in\nuke_H$. If $h\in K$, then, by the condition of the proposition, $h=1$. Thus, the stabilizers of the points of $\til$ in $K$ are trivial. Since $\til$ intersects every $H$-orbit, this implies that stabilizers in $K$ of all points of $\limg[H]$ are trivial.
\end{proof}

\subsection{Cut points}
We say that a point $\xi$ of a connected topological space $\M$ is a \emph{cut point} if the space $\M\setminus\{\xi\}$ is disconnected. We define the \emph{local degree} of a point $\xi$ of a locally connected space $\M$ as the supremum of the number of connected components of $U\setminus\{\xi\}$ over all connected neighborhoods $U$ of $\xi$. A point $\xi$ of local degree at least 2 is called a \emph{local cut point}.

\begin{thm}
\label{th:endsandcutpoints}
    Let $(G, \bim)$ be a contracting finitely generated self-replicating group. The following three numbers are finite and equal.
    \begin{enumerate}
        \item The supremum of $\esf(\Gammatilde_\xi)$ for $\xi\in\xo$.
        \item The supremum of $|\pi_0(\leaftilde_\xi\setminus\{\zeta\})|$ for $\xi\in\xo$ and $\zeta\in\leaftilde_\xi$.
        \item The supremum of the local degrees of points of $\limg$.
    \end{enumerate}
    
The same statement holds with $\Gammatilde_\xi$, $\leaftilde_\xi$, and $\limg$ replaced by $\Gamma_\xi$, $\leaf_\xi$, and $\lims$, respectively.
\end{thm}

\begin{proof}
We will prove the theorem for the case of graphs and leaves of germs, as it is the case needed for our paper. The case of orbital graphs and leaves is analogous.

Let us fix a set $L$ satisfying the conditions of Lemma 
\ref{lem:connectivityB}, and thus also Lemma \ref{lem:extendingpaths}, that will be invoked in this proof.  Since a set containing $L$ also satisfies the conditions of the lemma, we may assume that $L=\nuke^l$ for some $l$.

Arguments of the proof of Lemma~\ref{lem:extendingpaths} can be used to prove the following statement (see the proof of~\cite[Theorem~3.5.1]{Nek:book}).

\begin{lemma}
    \label{lem:connectedtile}
The tile $\til$ belongs to one connected component of $\til\cdot L$.
\end{lemma}

Denote by $D_1, D_2, D_3$ the numbers defined in items (1), (2), (3) of our theorem, respectively. Before we prove that they are finite, we allow them to be $\infty$ (regardless of the corresponding infinite cardinality).

Let us prove at first the inequality $D_1\le D_2$. It is enough to prove that if
$\esf(\Gammatilde_{x_1x_2\ldots})\ge K$ for $x_1x_2\ldots \in\xo$ and a natural number $K$, then $D_2\ge K$.

Let $D\subset G$  be a finite set such that $\Gammatilde_{x_1x_2\ldots}\setminus [D, x_1x_2\ldots]$ has at least $K$ infinite components.

Denote
\[T_n=\{[g, x_1x_2\ldots]\colon g|_{x_1x_2\ldots x_n}\in\nuke\}.\]
Equivalently, it is the set of germs at $x_1x_2\ldots$ of the transformations \[x_1x_2\ldots x_n\xo\arr v\xo\colon x_1x_2\ldots x_nw\mapsto v h(w)\] for $v\in\alb^n$ and $h\in\nuke$.

For all $n$ large enough, we have $[D, x_1x_2\ldots]\subset T_n$. 
Therefore, there exist $K$ vertices $[g_{i, n}, x_1x_2\ldots]$, $i=1, 2, \ldots, K$, of $\Gammatilde_{x_1x_2\ldots}$ outside of $T_n$ that belong to different infinite components of $\Gammatilde_{x_1x_2\ldots}\setminus [D, x_1x_2\ldots]$ and such that they are adjacent to some vertices of $T_n$. We may assume that $g_{i, n}|_{x_1x_2\ldots x_n}\in\nuke^2$, since the vertices are adjacent to $T_n$. We also have $[g_{i, n}|_{x_1x_2\ldots x_n}, x_{n+1}x_{n+2}\ldots]\notin [\nuke, x_{n+1}x_{n+2}\ldots]$, since they do not belong to $T_n$. 

Let us shift the sequence $x_1x_2\ldots$ by $\sigma^{-n}$, and denote
\[x_1x_2\ldots x_n=y_{-n, n}y_{-n+1, n}\ldots y_{-1, n},\qquad x_{n+1}x_{n+2}\ldots=y_{1, n}y_{2, n}\ldots\]
and 
\[
g_{i, n}(x_1x_2\ldots) = a_{-n, n}^{(i)}a_{-n+1, n}^{(i)}\ldots a_{-1, n}^{(i)}\;a_{1, n}^{(i)}a_{2, n}^{(i)}\ldots.
\]

Since $\alb$ and $\nuke^2$ are finite, there exists an increasing sequence $n_k$ such that, for every $i, j$, the sequences
\[y_{i, n_k},\quad a_{i, n_k}^{(j)},\quad g_{j, n_k}|_{x_1x_2\ldots x_{n_k}}\]
are eventually constant. Let $y_i, a_i^{(j)}, h_j$ be their limits. Since \[g_{j, n}|_{x_1x_2\ldots x_n}(y_{1, n}y_{2, n}\ldots)=a_{1, n}^{(j)}a_{2, n}^{(j)}\ldots,\] we have $h_j(y_1y_2\ldots)=a_1^{(j)}a_2^{(j)}\ldots$. 

Consider the points  of $\leaftilde_{y_1y_2\ldots}$
\begin{eqnarray*}
\zeta &=& \ldots y_{-2}y_{-1}\otimes [1, y_1y_2\ldots],\\
\eta_j &=& \ldots a_{-2}^{(j)}a_{-1}^{(j)}\otimes [h_j, y_1y_2\ldots].
\end{eqnarray*}
(Here we use the equality sign informally, meaning that the right-hand side is a sequence representing the point on the left-hand side.)

Let us show that for every $j$ the germ $[h_j, y_1y_2\ldots]$ is not equal to the germ at $y_1y_2\ldots$ of an element of $\nuke$. Suppose that, on the contrary, $h'\in\nuke$ is such that $[h_j, y_1y_2\ldots]=[h', y_1y_2\ldots]$. Then there exists $m$ such that $h_j(y_1y_2\ldots y_m)=h'(y_1y_2\ldots y_m)$ and $h_j|_{y_1y_2\ldots y_m}=h'|_{y_1y_2\ldots y_m}$. For all $k$ large enough, we have $y_1y_2\ldots y_m=x_{n_k+1}x_{n_k+2}\ldots x_{n_k+m}$ and $h_j=g_{j, n_k}|_{x_1x_2\ldots x_{n_k}}$. Therefore, for all such $k$, we have $g_{j, n_k}|_{x_1x_2\ldots x_{n_k}}(x_{n_k+1}x_{n_k+2}\ldots x_{n_k+m})=h'(x_{n_k+1}x_{n_k+2}\ldots x_{n_k+m})$ and $g_{j, n_k}|_{x_1x_2\ldots x_{n_k}}|_{x_{n_k+1}x_{n_k+2}\ldots x_{n_k+m}}=h'|_{x_{n_k+1}x_{n_k+2}\ldots x_{n_k+m}}$, hence the germs of $h'$ and $g_{j, n_k}|_{x_1x_2\ldots x_{n_k}}$ in $x_{n_k+1}x_{n_k+2}\ldots$ coincide. But this contradicts the choice of the elements $g_{j, n}$.

Since $[h_j, y_1y_2\ldots]\notin[\nuke, y_1y_2\ldots]$, we have $\zeta\notin\{\eta_1, \eta_2, \ldots, \eta_K\}$.  Let us show that the $\eta_j$ belong to pairwise different connected components of $\leaftilde_{y_1y_2\ldots}\setminus\{\zeta\}$.  It is enough to show that any two of them, e.g., $\eta_1$ and $\eta_2$, belong to different connected components. 

Suppose that it is not true. Since $\leaftilde_{y_1y_2\ldots}$ is locally path-connected (by Proposition \ref{prop:locconnected}), the connected components of $\leaftilde_{y_1y_2\ldots}\setminus\{\zeta\}$ are its path-connected components, so there exists a path $\alpha$ in $\leaftilde_{y_1y_2\ldots}$ connecting $\eta_1$ to $\eta_2$ that does not pass through $\zeta$. 

Fix some $n\ge 1$. Consider the set of all tiles of the level $n$ intersecting the path $\alpha$. It includes the tiles 
\[
\til\otimes a_{-n}^{(j)}a_{-n+1}^{(j)}\ldots a_{-1}^{(j)}\cdot h_j\otimes [1, y_1y_2\ldots]\ni\eta_j
\]
for $j=1, 2$.

The graph spanned by this set in the adjacency graph of $n$th level tiles is connected (otherwise we get a disconnection of the path $\alpha$), hence we can find a finite chain $\gamma$ of the $n$th level tiles intersecting $\alpha$ connecting the tiles  $\til\otimes a_{-n}^{(j)}a_{-n+1}^{(j)}\ldots a_{-1}^{(j)}\cdot h_j\otimes [1, y_1y_2\ldots]$.

  By compactness, if $n$ is large enough, each tile in $\gamma$ is disjoint from the neighborhood 
  \[\til_n:=\bigcup_{s\in \nuke^{l+1}}\til\otimes s\cdot y_{-n}y_{-n+1}\ldots y_{-1}\otimes [1, y_1y_2\ldots]\] of $\zeta$. Let us fix such $n=:N_1$. 

The set of tiles of the form $\til\otimes v\otimes [h, y_1y_2\ldots]$ for $v\in\alb^{N_1}$ and $h\in\nuke^2$ is finite, so, by Proposition~\ref{pr:convergenceofgraphs}, for all $N_2$ large enough the map $v\otimes g\cdot y_1y_2\ldots y_{N_2}\mapsto \til\otimes v\otimes [g, y_1y_2\ldots]$ from the graph induced by  the subset
$\alb^{N_1}\otimes \nuke^2\otimes \alb^{N_2}$ of $\Xi_{N_1+N_2}$ to the adjacency graph of the tiles of level $N_1$ of $\leaftilde_{y_1y_2\ldots}$ is an isomorphic embedding.

The preimage of the path $\gamma$ under this isomorphism is a path $v_1, v_2, \ldots, v_m$ in $\Xi_{N_1+N_2}$ with the endpoints 
\begin{align*}
v_1 &=a_{-N_1}^{(1)}a_{-N_1+1}^{(1)}\ldots a_{-1}^{(1)}\cdot h_1\cdot y_1y_2\ldots y_{N_2},\\
v_m &=a_{-N_1}^{(2)}a_{-N_1+1}^{(2)}\ldots a_{-1}^{(2)}\cdot h_2\cdot y_1y_2\ldots y_{N_2}.
\end{align*}

By Lemma~\ref{lem:extendingpaths}, for all $k$ large enough, there is a path $\gamma'$ in $\Xi_{n_k+N_2}$ from $g_{1, n_k}\cdot x_1x_2\ldots x_{n_k+N_2}$ to $g_{2, n_k}\cdot x_1x_2\ldots x_{n_k+N_2}$ staying inside the set $\bigcup_{i=1}^m\alb^{n_k-N_1}\nuke^l\cdot v_i$. Since the points $v_i$ are outside the set $\nuke^{l+1}\cdot y_{-N_1}y_{-N_1+1}\ldots y_{-1}$, the path $\gamma'$ does not intersect the set $D\cdot \alb^{n_k}\subset\alb^{n_k-N_1}\nuke\cdot y_{-N_1}y_{-N_1+1}\ldots y_{-1}$ for all $k$ large enough. Mapping the path $\gamma'$ to $\Gammatilde_{x_1x_2\ldots}$, we get a contradiction with the condition that the vertices $[g_{1, n}, x_1x_2\ldots]$ and $[g_{2, n}, x_1x_2\ldots]$ belong to different connected components of $\Gammatilde_{x_1x_2\ldots}\setminus [D, x_1x_2\ldots]$. This finishes the proof of the inequality $D_1\le D_2$.

Let us show that $D_1$ is finite. Let $D\subset G$ be a finite set, and let $n$ be such that $[D, x_1x_2\ldots]\subset T_n=\{[g, x_1x_2\ldots]\colon g|_{x_1x_2\ldots x_n}\in\nuke\}$. Let $L=\nuke^l$ be a set satisfying the conditions of Lemma~\ref{lem:connectedtile}. 

For any $h\in\nuke^{l+1}\setminus\nuke^l$ consider the set $T_{h, n}$ of germs of the form $v_1w\mapsto v_2h(w)$ for $v_1=x_1x_2\ldots x_n$ and $v_2\in\alb^n$. By Lemma~\ref{lem:connectedtile}, any two vertices of the set $T_{h, n}$ are connected in $\Gammatilde_{x_1x_2\ldots}$ by a path disjoint from $T_n\supset [D, x_1x_2\ldots]$. Consequently, each of these sets is contained in at most one connected component of $\Gammatilde_{x_1x_2\ldots}\setminus [D, x_1x_2\ldots]$. On the other hand, every infinite connected component contains some vertex $[h, x_1x_2\ldots]$ such that $h\in\nuke^{l+1}\setminus\nuke^l$. Consequently, the number of infinite connected components of $\Gammatilde_{x_1x_2\ldots}\setminus [D, x_1x_2\ldots]$ is at most $|\nuke^{l+1}\setminus\nuke^l|$.

The inequality $D_2\le D_3$ follows directly from Proposition~\ref{pr:localhomeomorphism}:  indeed since germ-leaves are path-connected, any path-connected component of $ \leaftilde_\xi\setminus \{\zeta\}$ accumulates on $\zeta$; thus any global cut point  $\zeta$ is also a local cut point of local degree at least $|\pi_0(\leaftilde_\xi\setminus \{\zeta\})|$.  

It remains to show that $D_3\le D_1$. Let  $\xi\in\limg$ be a point of local degree $\ge K$ for some natural number $K$. 

By applying an element of $G$, we may assume that it belongs to the tile $\til$, i.e., can be represented by a sequence $\ldots x_2x_1\cdot 1$ for $x_i\in\alb$. Let $U$ be a connected neighborhood of $\xi$ such that $|\pi_0(U\setminus\{\xi\})|\ge K$. It is enough to show that $D_1\ge K$, i.e., that there exists a point $\zeta\in\xo$ such that $\esf(\Gammatilde_\zeta)\ge K$.

Denote
\[x_n\ldots x_2x_1=y_{1, n}y_{2, n}\ldots y_{n, n},\]
and pass to a sequence $n_k\to\infty$ such that $y_i=\lim_{k\to\infty}y_{i, n_k}$ for every $i$. Moreover, we may assume that the sequence of subgroups $G_{x_{n_k}\ldots x_2x_1}^0$ converges to a subgroup $H$ in the space of subgroups of $G$. Equivalently, the rooted graphs $(\Xi_{n_k}, x_{n_k}\ldots x_2x_1)$ converge to the Schreier graph $\Gamma_{G/H}$.

By Lemma \ref{l-semi-continuous}, we have $G^0_{\zeta}\le H\le G_\zeta$ for $\zeta=y_1y_2\ldots$, and hence $G^0_\zeta$ has finite index in $H$. Then, by Lemma~\ref{lem:endsofcoverings}, we have $\esf(\Gammatilde_\zeta)\ge\esf(\Gamma_{G/H})$. Consequently, it is enough to prove that $\esf(\Gamma_{G/H})\ge K$.

Suppose that, on the contrary, $\esf(\Gamma_{G/H})<K$. Let $R$ be an arbitrary large radius. Then the ball $B_{\Gamma_{G/H}}(H, R)$ is isomorphic to the ball $B_{\Xi_{n_k}}(x_{n_k}\ldots x_2x_1, R)$ of $\Xi_{n_k}$ for all $k$ large enough. Consider the set of tiles $ \til\otimes \nuke^R\cdot x_{n_k}\ldots x_2x_1$. Its union is a neighborhood of $\xi$. The adjacency graph of this set of tiles is isomorphic to $B_{\Xi_{n_k}}(x_{n_k}\ldots x_2x_1, R)$. For all $k$ large enough (when $R$ is fixed), the union of 0th level tiles $\til\otimes \nuke^R\cdot x_{n_k}\ldots x_2x_1$ is contained in $U$.  

Let $L=\nuke^l$ be as in Lemma~\ref{lem:connectedtile}. Suppose that $g_1, g_2\in G$ are such that there exists a path in $B_{\Xi_{n_k}}(x_{n_k}\ldots x_2x_1, R-l)$ connecting $g_1\cdot x_{n_k}\ldots x_2x_1$ to $g_2\cdot x_{n_k}\ldots x_2x_1$ and not passing through the set $\nuke^{l+1}\cdot x_{n_k}\ldots x_1$. Then, by Lemma~\ref{lem:connectedtile}, the tiles $\til\otimes g_i\cdot x_{n_k}\ldots x_2x_1$ are contained in one connected component of $\bigcup_{g\in \nuke^R\setminus\nuke}\til\otimes g\cdot x_{n_k}\ldots x_2x_1$. The latter set does not contain $\xi$, while the set $\bigcup_{g\in\nuke}\til\otimes g\cdot x_{n_k}\ldots x_2x_1$ converges to $\xi$ as $k\to\infty$. 

The set $U\setminus\{\xi\}$ has at least $K$ connected components, each of which has a non-empty interior and accumulates on $\xi$. Consequently, the graph $B_{\Xi_{n_k}}(x_{n_k}\ldots x_2x_1, R-l)\setminus\nuke^{l+1}\cdot x_{n_k}\ldots x_2x_1$ has $K$ arbitrarily large connected components. Since we assume that $\Gamma_{G/H}$ has fewer than $K$ ends, the graph $\Gamma_{G/H}\setminus \nuke^{l+1} H$ has fewer than $K$ infinite connected components and finitely many finite ones. Let $M$ be such that every finite connected component has fewer than $M$ vertices. Since balls of graphs are connected, each connected component of $B_{\Xi_{n_k}}(x_{n_k}\ldots x_2x_1, R-l)\setminus\nuke^{l+1}\cdot x_{n_k}\ldots x_2x_1$ is adjacent to a vertex of $\nuke^{l+1}\cdot x_{n_k}\ldots x_2x_1$. Since we have at least $K$ arbitrarily large connected components, we can find $K$ connected graphs consisting of $M$ vertices, containing a vertex adjacent to $\nuke^{l+1}\cdot x_{n_k}\ldots x_2x_1$, and belonging to different connected components of $B_{\Xi_{n_k}}(x_{n_k}\ldots x_2x_1, R-l)\setminus\nuke^{l+1}\cdot x_{n_k}\ldots x_2x_1$. 

After passing to a subsequence, we may assume that these $K$ graphs have the same image in $\Gamma_{G/H}$. This will give us $K$ connected subgraphs in $\Gamma_{G/H}\setminus\nuke^{l+1} H$, each consisting of $M$ vertices and such that they are not connected inside $B_{\Gamma_{G/H}}(H, R-l)\setminus\nuke^{l+1}\cdot H$ for any $R$. But this implies that they belong to different connected components of $\Gamma_{G/H}\setminus \nuke^{l+1} H$. Since they have more vertices than any finite connected component of $\Gamma_{G/H}\setminus \nuke^{l+1} H$, they belong to $K$ different infinite components, which is a contradiction.
\end{proof}

Corollary \ref{c-intro-FW} follows by combining Theorem \ref{th:endsandcutpoints} with Theorem \ref{t-multi-ended-graphs}. More precisely:

\begin{cor}
    Let $(G, \bim)$ be a finitely generated contracting self-replicating branch group. Let $\alpha_G=\sup_{H\le G}\esf(\Gamma_{G/H})$, and let $\beta_G$ be the maximum local degree of points in $\limg$. Then we have $\alpha_G=\beta_G$ if $G$ is just-infinite, and $\alpha_G=\max(\beta_G, 2)$ otherwise. In particular, $G$ has Property FW if and only if it is just-infinite and $\limg$ has no local cut point.

\end{cor}

\begin{proof}
From Theorems \ref{t-multi-ended-graphs} and \ref{th:endsandcutpoints}, it is clear that the supremum of the number of ends of all {faithful} Schreier graphs of $G$ is equal to $\beta_G$. If $G$ is just-infinite, that number is equal to $\alpha_G$. Otherwise $G$ also admits infinite Schreier graphs induced from its infinite virtually abelian quotients (recall Proposition \ref{p-Grigorchuk}). But for every infinite virtually abelian group $Q$, we have $\sup_{H\le Q} \esf(\Gamma_{Q/H})=2$, see \cite[\S6.B]{Cor-FWsurvey}.
\end{proof}
\subsection{Bounding cardinal-definite functions}
Let $(G, \bim)$ be a self-similar group with basis $\alb$. We use the notation 

\[\mathrm{supp}_{\alb^n}(g):=\{x\in \alb^n \colon g(x)\neq x\}.\]

\begin{prop} \label{p-contracting-cardinal-definite}
Let $(G, \bim)$ be a finitely generated contracting self-replicating group. Fix a basis $\alb$ and let $\nuke$ be the corresponding nucleus. Let $A\subset \Gammatilde_\xi$ be a commensurated subset of the graph of germs of some $\xi\in \xo$, and $\nu_A(\cdot)$ be the associated cardinal-definite function.  Then there exists $C\ge 0$ such that for every $n\ge 1$ we have 
\[\nu_A(g)\leq C|\mathrm{supp}_{\alb^n}(g)|+ C\sum_{v\in \alb^n} \ell_\nuke(g|_v),\]
for every $g\in G$, where $\ell_\nuke(\cdot)$ is the word-length in the generating set $\nuke$. 
\end{prop}
\begin{proof}
 Let $L=\nuke^l$ be given by Lemma \ref{lem:connectivityB}, which therefore also satisfies Lemma \ref{lem:connectedtile}. Suppose that $\xi=x_1x_2\cdots$. Fix $n\ge 1$. We apply the shift $\sigma^{-n}$ and consider the germ-leaf $\leaftilde_{\sigma^{-n}(\xi)}$, for $\sigma^{-n}(\xi)=x_{n+1}x_{n+2}\cdots$. Then $\Gammatilde_{\sigma^{-n}(\xi)}$ is the adjacency graph of level 0 tiles $\til\otimes [g, \sigma^{-n}(\xi)]$ while $\Gammatilde_\xi$ is the adjacency graph of level $n$ tiles $\til\otimes v\otimes[g, \sigma^{-n}(\xi)]$ for $v\in \alb^n$ and $g\in G$. Accordingly, we identify the set of vertices of $\Gammatilde_{\sigma^{-n}(\xi)}$ with the set of germs $[G, \sigma^{-n}(\xi)]$, and the set of vertices of $\Gammatilde_\xi$ with $\bim^{\otimes n}\otimes [G, \sigma^{-n}(\xi)]$  (the latter is in explicit bijection with $[G, \xi]$, namely $v\otimes[g, \sigma^{-n}(\xi)]$ for $v\in \alb^n$ and $g\in G$ corresponds to the germ at $\xi$ of the map $y_1y_2\cdots \mapsto vg(y_{n+1}y_{n+2}\cdots)$).

 Let $A\subset \Gammatilde_\xi$ be the commensurated subset in the statement, and $\partial A:=\{a\in A\colon \nuke \cdot a\not \subset A\}$. Consider the corresponding subset of $\Gammatilde_{\sigma^{-n}(\xi)}$
 \[\sigma^{-n}(\partial A):= \{[g, \sigma^{-n}(\xi)] \colon \exists v\in \alb^n,  v\otimes [g, \sigma^{-n}(\xi)] \in \partial A\}\subset \Gammatilde_{\sigma^{-n}(\xi)}.\]
Clearly  $|\sigma^{-n}(\partial A)|\leq |\partial A|$. 

Fix $g\in G$ and let us estimate the cardinality of the set $D:=\{a\in A \colon g\cdot a\notin A\}$. For this, we decompose $D$ as $D=D_0\cup (\bigcup_{v\in \alb^n} D_v)$,
where we set $D_0=\{v\otimes [h, \sigma^{-n}(\xi)] \in D \colon v\in \alb^n, g|_v=e\}$,  and  for every $v\in \alb^n$ we set $D_v=\{v\otimes [h, \sigma^{-n}(\xi)]\in D \colon g|_v\neq e\}$. 

Fix $v\otimes [h, \sigma^{-n}(\xi)]\in D_0$. We have 
\[g\cdot (v\otimes [h, \sigma^{-n}(\xi)])=g(v)\otimes g|_v\cdot [h, \sigma^{-n}(\xi)]=g(v)\otimes  [h, \sigma^{-n}(\xi)].\] By definition of $D$, no point of $D$ can be fixed by $g$, hence necessarily $v\in \mathrm{supp}_{\alb^n}(g)$. On the other hand the two $n$-level tiles $ \til\otimes v\otimes [h, \sigma^{-n}(\xi)]$ and $\til\otimes g(v)\otimes  [h, \sigma^{-n}(\xi)]$ are contained in the same 0-level tile $\til\otimes [h, \sigma^{-n}(\xi)]$. Hence by Lemma \ref{lem:connectedtile} they can be connected by a path $\alpha$ in $\leaftilde_{\sigma^{-n}(\xi)}$ such that the  $0$-tiles visited by $\alpha$  describe a path $\gamma\subset B_{\Gammatilde_{\sigma^{-n}(\xi)}}([h, \sigma^{-n}(\xi)], l)$. On the other hand, the $n$-tiles visited by $\alpha$ span a path in $\Gammatilde_\xi$ starting in $v\otimes [h, \sigma^{-n}(\xi)]\in A$ and ending in $g(v)\otimes [h, \sigma^{-n}(\xi)]\in A^c$. This path must pass through $\partial A$, and hence $\gamma$ must pass through $\sigma^{-n}(\partial A)$. We deduce that $ [h, \sigma^{-n}(\xi)]\in B_{\Gammatilde_{\sigma^{-n}(\xi)}}(\sigma^{-n}(\partial A), l)$. Hence
\begin{equation}|D_0|\leq |\mathrm{supp}_{\alb^n}(g)|\cdot |B_{\Gammatilde_{\sigma^{-n}(\xi)}}(\sigma^{-n}(\partial A), l) |\leq C_1 |\mathrm{supp}_{\alb^n}(g)|,\label{e-D0}\end{equation}
where $C_1=|\partial A|\cdot |\nuke |^l$.

Now fix $v\in \alb^n$ and let us estimate the cardinality of $D_v$. Write $g|_v=s_m\cdots s_1$, with $s_i\in \nuke$ and $m=\ell_\nuke(g|_v)$. Let $v\otimes\rho\in D_v$, where $\rho=[h, \sigma^{-n}(\xi)]$ for some $h\in G$. Then $\gamma=(\rho, s_1\cdot \rho,\ldots, s_m\cdots s_1\cdot \rho)$ is a path in $\Gammatilde_{\sigma^{-n}(\xi)}$ from $\rho$ to $g|_v\cdot \rho$, corresponding to a chain of 0-tiles in $\leaftilde_{\sigma^{-n}(\xi)}$. The first 0-tile in this chain is  $\til\otimes \rho$, so it  contains the $n$-tile $\til \otimes v\otimes \rho$, while the last one,  $\til\otimes g|_v\cdot \rho$, contains the $n$-tile $\til\otimes g(v)\otimes g|_v\cdot \rho$.  By Lemma \ref{lem:connectedtile}, there exists a path $\alpha$ in $\leaftilde_{\sigma^{-n}(\xi)}$ starting at a point in $\til \otimes v\otimes \rho$ and ending at a point in $\til\otimes g(v)\otimes g|_v\cdot \rho$, and such that the 0-tiles met by $\alpha$ define a path $\gamma_1\subset \Gammatilde_{\sigma^{-n}(\xi)}$ that stays inside an $l$-neighbourhood of $\gamma$. The $n$-tiles spanned by $\alpha$ define a path in $\Gammatilde_\xi$ that starts at  $v\otimes \rho\in A$ and ends at $g(v)\otimes g|_v\cdot \rho=g\cdot( v\otimes \rho)\notin A$, and hence must visit $\partial A$. It follows that the path $\gamma_1$ visits $\sigma^{-n}(\partial A)$; and hence $\gamma$ visits $B_{\Gammatilde_{\sigma^{-n}(\xi)}}(\sigma^{-n}(\partial A), l)$, i.e. $s_i\cdots s_1\rho \in B_{\Gammatilde_{\sigma^{-n}(\xi)}}(\sigma^{-n}(\partial A), l)$ for some $i=0, \ldots, m$ (where for $i=0$ we agree that $s_i\cdots s_1=e$). Hence 
\[\rho\in\bigcup_{i=0}^m s_1^{-1}\cdots s^{-1}_i\cdot B_{\Gammatilde_{\sigma^{-n}(\xi)}}(\sigma^{-n}(\partial A), l).\] 

Since $v\otimes \rho$ was an arbitrary element of $D_v$, we deduce that 
\begin{equation} |D_v|\leq (m+1)\cdot |B_{\Gammatilde_{\sigma^{-n}(\xi)}}(\sigma^{-n}(\partial A), l)|\leq 2C_1 \ell_\nuke(g|_v), \label{e-Dv}\end{equation}
where $C_1=|\partial A|\cdot |\nuke |^l$ as above, and we used that $m=\ell_\nuke(g|_v)\ge 1$. 

The estimates \eqref{e-D0}--\eqref{e-Dv} together give 
\[|D|=|D_0|+\sum_{v\in \alb^n} |D_v| \leq C_2|\mathrm{supp}_{\alb^n}(g)|+ C_2\sum_{v\in \alb^n} \ell_\nuke(g|_v).\]
Repeating the same reasoning, we obtain the analogous bound for the cardinality of $D^\ast=\{a\in A^c\colon g\cdot a\in A\}$, and hence for $\nu_A(g)=|D|+|D^\ast|$. \qedhere
\end{proof}
Therefore Corollary \ref{c-cardinal-definite}, combined with Proposition \ref{p-contracting-cardinal-definite}, gives the following, which is a more precise form of Corollary \ref{c-intro-not-PW}. 
\begin{cor} \label{c-cardinal-definite-contracting}
Let $(G, \bim)$ be a finitely generated, faithful, contracting self-replicating branch group, with basis $\alb$ and nucleus $\nuke$. 
Let $\nu_A$ be a cardinal-definite function on $G$. Then there exists $C>0$ and $n_0\ge 1$ such that
\[\nu_A(g)\leq C\sum_{v\in \alb^n}\ell_\nuke(g|_v) \]
for every $n\ge n_0$ and every $g\in \rist(n)'$. In particular, if $G$ is regular branch, then $G$ does not have Property PW.

\end{cor}

\begin{proof}
    Let $\nu_A$ be a cardinal-definite function. Let $\phi_i\colon G\dashrightarrow \Z$ be the finitely many virtual homomorphisms given by Corollary \ref{c-cardinal-definite}. By Lemma \ref{l-double-comm}, for every large enough $n$, we have $\rist(n)'\le \ker \phi_i$ for all $i$. Furthermore, for all $g\in \rist(n)$,  the term $|\mathrm{supp}_{\alb^n}(g)|$ vanishes. Hence the bound follows immediately from Corollary \ref{c-cardinal-definite} and Proposition \ref{p-contracting-cardinal-definite}.  If $G$ is regular branch, then $\rist(n)'$ contains $(K')^{\alb^n}$ for some fixed finite index subgroup $K\le G$. Thus if we fix $k\in K'$ non-trivial, we can find $g_n\in \rist(n)'$ such that $g_n|_v=k$ for some $v\in \alb^n$ and $g_n|_w=e$ for every $v\neq w\in \alb^n$. Thus $\nu_A(g_n)$ stays bounded, although the sequence $(g_n)$ takes infinitely many values. \qedhere
\end{proof}
The first Grigorchuk group is contracting, self-replicating, and regular branch \cite{Gri-branch}, hence we have the following corollary.  
\begin{cor} \label{c-Grigorchuk-PW}
    The first Grigorchuk group does not have Property PW. 
\end{cor}

\subsection{Iterated monodromy groups}
\label{s-IMG}
Contracting self-similar groups appear naturally as iterated monodromy groups of locally expanding covering maps. Let $f\colon\M\arr\M$ be a covering map, where $\M$ is a compact path-connected metric space.  For $t\in\M$, the \emph{tree of preimages} is the graph with the set of vertices equal to the formal disjoint union $T_t=\bigsqcup_{n\ge 0}f^{-n}(t)$, where $z\in f^{-n}(t)$ is connected by an edge to $f(z)\in f^{-(n-1)}(t)$. If $\gamma$ is a path in $\M$ starting in $t_1$ and ending in $t_2$, then we define $\gamma(z)$ for $z\in f^{-n}(t_1)$ as the end of the lift $\gamma_z$ of $\gamma$ by the covering $f^n\colon\M\arr\M$ to a path starting in $z$. Then $z\mapsto\gamma(z)$ is an isomorphism $S_\gamma\colon T_{t_1}\arr T_{t_2}$ \emph{defined by $\gamma$}. In particular, every element $[\gamma]$ of the fundamental group $\pi_1(\M, t)$ defines an automorphism $S_\gamma$ of $T_t$. The set of all such automorphisms is the \emph{iterated monodromy group} $\mathop{\mathrm{IMG}}(f)$. The iterated monodromy group has a natural structure of a self-similar group defined by the biset $\bim_f$ equal to the set of isomorphisms $S_\ell\colon T_t\arr T_z$ for all $z\in f^{-1}(t)$ and all paths $\ell$ connecting $t$ to $z$. The left and right actions of $\mathop{\mathrm{IMG}}(f)$ on $\bim_f$ are by post- and pre-compositions, respectively.

More generally, one can define iterated monodromy groups of \emph{virtual endomorphisms} of topological spaces, i.e., pairs of continuous maps $f, \iota\colon\M_1\arr\M$, where $f$ is a finite degree covering. A lift of a point $t$ (or a path $\gamma$) by the $n$th iteration of the virtual endomorphism is a sequence $(t_1, t_2, \ldots, t_n)$ of points (or paths) in $\M_1$ such that $f(t_1)=t$ and $f(t_{i+1})=\iota(t_i)$. The iterated monodromy group and the biset are defined then in a similar way to the case of a covering map (corresponding to the case when $\iota$ is a homeomorphism). See \cite{Nek:combinatorialmodels} for more details on iterated monodromy groups of virtual endomorphisms (called \emph{topological automata} there).

We say that a self-similar group $G$ is \emph{regular} if for every $x_1x_2\ldots\in\xo$ and every $g\in G$ either $g(x_1x_2\ldots)\ne x_1x_2\ldots$ or there exists $n$ such that $g|_{x_1x_2\ldots x_n}=1$. If $G$ is regular, then $G_{x_1x_2\ldots}=G_{x_1x_2\ldots}^0$ for every $x_1x_2\ldots\in\xo$, and thus the graphs $\Gamma_{x_1x_2\ldots}$ and $\Gammatilde_{x_1x_2\ldots}$ coincide.

The following is proved in~\cite[Corollary~4.5.22 and Proposition~4.5.23]{nek:dyngroups}.

\begin{thm}
\label{th:imgmain}
    Let $f\colon\M\arr\M$ be a locally expanding covering map. Then $\mathop{\mathrm{IMG}}(f)$ is a regular contracting self-similar group such that the limit dynamical system $\mathsf{s}\colon\lims\arr\lims$ is topologically conjugate to $f\colon\M\arr\M$.

    Conversely, if $(G, \bim)$ is a regular contracting self-similar group, then $\mathsf{s}\colon\lims\arr\lims$ is a locally expanding covering map, and $\mathop{\mathrm{IMG}}(\mathsf{s})$ is isomorphic, as a self-similar group, to $(G, \bim)$.
\end{thm}

In the general (non-regular case) one has to consider \emph{virtual endomorphisms} of orbispaces (spaces locally described as quotients of topological spaces by finite groups). Then the analog of Theorem~\ref{th:imgmain} holds, when $\mathsf{s}\colon\lims\arr\lims$ is interpreted as a virtual endomorphism of an orbispace with underlying space $\lims$.
This orbispace is equivalent to the orbispace defined by the action of $G$ on the limit $G$-space $\limg$. See \cite[Chapter 5]{Nek:book}.

As an example of an application of Theorem~\ref{th:endsandcutpoints}, we have the following classification of iterated monodromy groups of complex rational functions whose graphs of germs and orbital graphs are one-ended.

We say that a complex rational function $f(z)\in\mathbb{C}(z)$ is \emph{post-critically finite} if the forward $f$-orbit of every critical point of $f$ is finite. We denote by $P_f$ the union of the orbits of the critical values of $f$. 

If $f(z)$ is post-critically finite, then we have a virtual endomorphism $f, \iota\colon\widehat{\mathbb{C}}\setminus f^{-1}(P_f)\arr\widehat{\mathbb{C}}\setminus P_f$, where $\iota$ is the identical embedding. Its iterated monodromy group is, by definition, the \emph{iterated monodromy group} of $f$, denoted $\mathop{\mathrm{IMG}}(f)$. Its limit dynamical system $\mathsf{s}\colon \lims\to \lims$ is topologically conjugate to the restriction of $f$ to its Julia set, by \cite[Theorem 5.5.3]{Nek:book}.

\begin{cor} \label{c-IMG}
    Let $f\in\mathbb{C}(z)$ be a post-critically finite rational function. The following conditions for the iterated monodromy group $\mathop{\mathrm{IMG}}(f)=(G, \bim)$ are equivalent:
    \begin{enumerate}
        \item The orbital graphs of the action of $G$ on the boundary of the tree are one-ended.
        \item The graphs of germs of the action of $G$ on the boundary of the tree are one-ended.
        \item The Julia set of $f$ is either the whole sphere or is homeomorphic to the Sierpi\'nski carpet.
    \end{enumerate}
\end{cor}

\begin{proof}
We have the following characterizations of the Sierpi\'nski carpet due to G.~Whyburn (see~\cite{Whyburn1958}).

\begin{thm}
\label{th:Whyburn}
    Let $\mathcal{C}$ be a compact connected subset of the sphere. Then the following conditions are equivalent.
    \begin{enumerate}
        \item The set $\mathcal{C}$ is homeomorphic to the Sierpi\'nski carpet.
        \item The set $\mathcal{C}$ has covering dimension 1, is locally connected, and has no local cut points.
        \item The complement of $\mathcal{C}$ is a union of a countable set of connected components $C_1, C_2, \ldots$ such that  \begin{itemize}
            \item the boundary $\partial C_i$ of each $C_i$ is a closed simple curve,
            \item $\partial C_i\cap\partial C_j=\emptyset$ for all $i\ne j$,
            \item $\bigcup_i\partial C_i$ is dense in $\mathcal{C}$,
            \item the diameters of $C_i$ go to 0 as $i\to\infty$.
        \end{itemize}
    \end{enumerate}
\end{thm}

Let $f$ be a post-critically finite rational function. Its Julia set is connected, locally connected, and if it is not the whole sphere $\widehat{\mathbb{C}}$, has covering dimension 1 (this follows, for example, from the fact that its Hausdorff dimension is strictly less than $2$, see~\cite[Corollary~6.2]{mcmullen}). Consequently, by Theorem~\ref{th:Whyburn}, if the Julia set has no local cut points, it is either $\widehat{\mathbb{C}}$, or homeomorphic to the Sierpi\'nski carpet.

Let us show that if the Julia set of $f$ is the sphere or a Sierpi\'nski carpet, then the limit $G$-space $\limg$ of its iterated monodromy group is locally homeomorphic to the sphere or the Sierpi\'nski carpet, respectively. This will finish the proof of the corollary.

For a post-critical point $z\in P_f$, denote by $\kappa(z)$ the least common multiple of local degrees of $f^n$ at all $x\in f^{-n}(z)$ for all $n\ge 1$. We have $\kappa(z)=\infty$ if and only if $f^n(z)$ is a periodic critical point for some $n\ge 0$.
The \emph{Thurston orbifold} $\mathcal{C}_f$ of $f$ is obtained by removing from $\widehat{\mathbb{C}}$ all points $z\in P_f$ such that $\kappa(z)=\infty$, and uniformizing the remaining points $z\in P_f$ by representing a neighborhood of $z$ as the orbispace of the action by rotations of a disc of a cyclic group of order $\kappa(z)$, where $z$ corresponds to the center of the disc. 

The fundamental group $H$ of the orbifold $\mathcal{C}_f$ admits a natural structure of a self-similar group $(H, \bim)$ such that its faithful quotient is the iterated monodromy group $G$ of $f$. In fact, the virtual endomorphism $f, \iota\colon\widehat{\mathbb{C}}\setminus f^{-1}(P_f)\arr \widehat{\mathbb{C}}\setminus P_f$ is completed to a virtual endomorphism of the Thurston orbifold. The group $(H, \bim)$ is contracting (see~\cite[\S4.7.1]{nek:dyngroups}), and its limit $H$-space $\limg[H]$ is the lift of the Julia set of $f$ to the universal covering $\rho\colon\mathcal{H}\arr\mathcal{C}_f$ of the Thurston orbifold. The universal covering is a simply connected Riemann surface, biholomorphic to the hyperbolic plane or to the Euclidean plane, on which $H$ acts properly by holomorphic automorphisms. Non-trivial elements of the nucleus of $H$ are non-trivial in the iterated monodromy group (see~\cite[Proposition~4.7.3]{nek:dyngroups}).

It follows from Proposition~\ref{pr:limgH} that the natural map $\limg[H]\arr\limg$ is a covering. In particular, $\limg$ is locally homeomorphic to the lift of the Julia set of $f$ to the universal covering of the Thurston orbifold. It follows that if the Julia set is the sphere, then the limit $G$-space $\limg$ is a one-dimensional complex manifold, so it has no local cut points. 

Suppose that the Julia set of $f$ is a carpet. The ramification points of the universal covering map $\rho\colon\mathcal{H}\arr\mathcal{C}_f$ are the points of $P_f$ on which the function $\kappa$ is finite and different from 1. Suppose that there exists a ramification point $z$ of $\rho$ on the boundary of a Fatou component of $f$. Then $z$ is the landing point of an internal ray in the Fatou component (i.e., the image of a radius of the unit disc under the B\"ottcher parametrization of the Fatou component, see~\cite[\S9]{milnor:complexdynamics}), and there exists a critical point $z_0$ and $n$ such that $f^n(z_0)=z$. Then $z_0$ is the landing point of more than one internal ray of Fatou components, hence it is a cut point in the Julia set. We get a contradiction with the condition that the Julia set has no local cut points. Consequently, the ramification points of $\rho$ are not on the boundaries of the Fatou components. Therefore, condition (3) of Theorem~\ref{th:Whyburn} implies that the lift of the Julia set of $f$ by $\rho$ is locally homeomorphic to the Sierpi\'nski carpet.
\end{proof}

It was proved in~\cite{MNZ} that if the Julia set of $f$ is not the whole sphere, then the iterated monodromy group of $f$ is amenable. It appears to be an interesting problem to determine for which $f$ the iterated monodromy group is branch and/or just-infinite.

\section{The Sierpi\'nski carpet group} \label{s-sierpinski}
	
Here we show, as an application of Theorem~\ref{t-multi-ended-graphs}, how the classical Sierpi\'nski carpet (shown on the left part of Figure \ref{fig:carpet}) can be used to construct an example of an amenable group with Property FW.

The Sierpi\'nski carpet group $G$ was  introduced in~\cite[\S6.3]{MNZ}. It is acting on the tree $\{1, 2, \ldots, 8\}^\ast$ and is generated by transformations $a, b, c, d$ satisfying the recursion
\begin{align*}
a &\mapsto(12)(67)(1, 1, a, 1, a, 1, 1, a),\\
b &\mapsto(46)(58)(b, b, b, 1, 1, 1, 1, 1),\\
c &\mapsto(23)(78)(c, 1, 1, c, 1, c, 1, 1),\\
d &\mapsto(14)(35)(1, 1, 1, 1, 1, d, d, d).
\end{align*}
We are using here the wreath recursion notation, where $g\mapsto\pi(g_1, g_2, \ldots, g_8)$ for $\pi\in\Sym_8$ replaces the relations 
\[g\cdot x=\pi(x)\cdot g_x\]
in the associated biset. The group $G$ is the corresponding faithful quotient, i.e., the group acting on the tree $\xs$ by the transformations given by $g(xw)=\pi(x)g_x(w)$. 
See Figure~\ref{fig:carpetmoore} for the graph $\Gamma_1$ of the action on the first level, where labels inside loops and double edges denote the acting element, and the labels outside denote the corresponding sections $g|_i$ (trivial if there is no label).

\begin{remark}
    The group $G$ is the iterated monodromy group of an 8-to-1 partial self-covering $p\colon C_1\to C$ of a unit square $C$, defined on the subset $C_1\subset C$ obtained by dividing $C$ in 9 equal squares and removing the middle one. The map $p$ stretches $C_1$ affinely by a factor $3$, and then folds it, so as to identify each of the 8 remaining squares with $C$. The iterate $p^n$ is thus defined on the natural $n$th step approximation $C_n\subset C$ of the Sierpi\'nski carpet. However, the covering map $p$ is branched over the boundary of $C$, so in order to define its iterated monodromy group one needs to endow $C$ with an appropriate structure of \emph{orbispace} that makes $p$ a honest covering (of orbispaces) \cite{Nek:book}.    In order to avoid recalling the relevant language, we chose to define $G$ directly by its wreath recursion.
\end{remark} 
\begin{figure}
    \centering
    \includegraphics{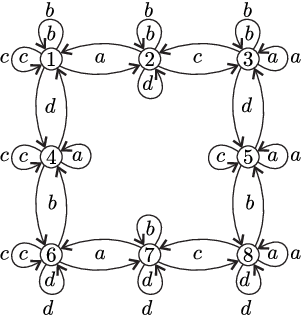}
    \caption{The first-level Schreier graph of the Sierpi\'nski carpet group $G$.    \label{fig:carpetmoore}
}
\end{figure}

\begin{prop}
Let $H$ be the group given by the presentation \[\langle a, b, c, d | a^2, b^2, c^2, d^2, (ab)^6, (bc)^6, (cd)^6, (da)^6\rangle.\] Then the above wreath recursion defines a contracting self-similar group $(H, \bim)$ with the nucleus $\nuke=\langle a, b\rangle\cup\langle b, c\rangle\cup\langle c, d\rangle\cup\langle d, a\rangle$ consisting of 41 elements.
\end{prop}

\begin{proof}
The above formulas obviously define a wreath recursion on the free group on $\{a, b, c, d\}$.  Let us first check that the wreath recursion remains well-defined on its quotient $H_1=\langle a, b, c, d | a^2, b^2, c^2, d^2\rangle$. We have
\begin{align*}
a^2 &\mapsto (1, 1, a^2, 1, a^2, 1, 1, a^2),\\
b^2 &\mapsto (b^2, b^2, b^2, 1, 1, 1, 1, 1),\\
c^2 &\mapsto (c^2, 1, 1, c^2, 1, c^2, 1, 1),\\
d^2 &\mapsto (1, 1, 1, 1, 1, d^2, d^2, d^2).
\end{align*}
Consequently, if we apply the wreath recursion to any element $g$ of the free group generated by $a, b, c, d$ such that $g$ is trivial in $H$, we will get $(g_1, g_2, \ldots, g_8)$, where each $g_i$ is trivial in $H_1$. Consequently, the wreath recursion is well-defined on $H_1$. 

Now if in $H_1$, we have (recall that the action on $\xs$ is from the left)
\begin{multline*}
ab\mapsto (12)(67)(1, 1, a, 1, a, 1, 1, a)(46)(58)(b, b, b, 1, 1, 1, 1, 1)=\\
(12)(476)(58)(1, 1, a, 1, a, 1, 1, a)(b, b, b, 1, 1, 1, 1, 1)=\\
(12)(476)(58)(b, b, ab, 1, a, 1, 1, a),
\end{multline*}
hence
\begin{multline*}
    (ab)^2\mapsto (12)(476)(58)(b, b, ab, 1, a, 1, 1, a)(12)(476)(58)(b, b, ab, 1, a, 1, 1, a)=\\
    (467)(b, b, ab, 1, a, 1, 1, a)(b, b, ab, 1, a, 1, 1, a)=\\
    (467)(1, 1, (ab)^2, 1, 1, 1, 1, 1),
\end{multline*}
so that
\[(ab)^6\mapsto (1, 1, (ab)^6, 1, 1, 1, 1, 1).\]

Similarly, 
\begin{align*}
    (bc)^6 &\mapsto ((bc)^6, 1, 1, 1, 1, 1, 1, 1),\\
    (cd)^6 &\mapsto (1, 1, 1, 1, 1, (cd)^6, 1, 1),\\
    (da)^6 &\mapsto (1, 1, 1, 1, 1, 1, 1, (da)^6).
\end{align*}

So the same reasoning as above shows that the wreath recursion is well-defined on $H$.

It is checked directly, that for every triple $h_1, h_2, h_3$ of pairwise different generators of $H$, we have $(\langle h_1, h_2\rangle\cdot h_3)|_x\subset\nuke$ for every $x\in\alb$. This implies that $\nuke$ contains the nucleus. It is also easy to check that every element of $\nuke$ is a section of an element of $\nuke$, hence $\nuke$ is the nucleus.
\end{proof}

Note that it follows from the above computation that the elements $ab, bc, cd, da$ of $G$ are of order $6$. Consequently, all non-trivial elements of the nucleus of $H$ are also non-trivial elements of the nucleus of $G$, so that the nucleus $\langle a, b\rangle\cup\langle b, c\rangle\cup\langle c, d\rangle\cup\langle d, a\rangle$ of $G$ also has 41 elements.

\begin{prop}
    The abelianization of $G$ is isomorphic to $(\Z/2\Z)^4$.
\end{prop}

\begin{proof}
Let $P\colon H\arr (\Z/2\Z)^4$ be the epimorphism mapping the generators $a, b, c, d$ of $H$ to the generators of $(\Z/2\Z)^4$. It follows from the presentation of $H$ that it is well defined.

It follows from the wreath recursion that we have $P(g)=P(g|_1)+P(g|_2)+\cdots+P(g|_8)$ for every $g\in H$ (indeed, the formula on the right-hand side of the equality also defines a homomorphism, to $(\Z/2\Z)^4$, and it coincides with $P$ on $\{a, b, c, d\}$). 
Consequently, $P(g)=\sum_{v\in\alb^n} P(g|_v)$ for every $n$. Suppose that $g$ belongs to the kernel of the epimorphism $H\arr G$. Since the wreath recursion is contracting on $H$ and non-trivial elements of the nucleus of $H$ are non-trivial in $G$, there exists $n$ such that all sections $g|_v$ for $v\in\alb^n$ are trivial in $H$. Consequently, $P(g)=0$. It follows that $P$ induces a well defined epimorphism $P\colon G\arr (\Z/2\Z)^4$, which descends to an isomorphism on the abelianisation $G/G'$ since $G$ is generated by four involutions.
\end{proof}

When talking about the group $\aut\xs$, we identify any element $\pi\in\Sym(\alb)$ with the automorphism $xv\mapsto \pi(x)v$, and call such automorphisms \emph{rooted}. Respectively, the action of a subgroup of $\Sym(\alb)$ on $\xs$ by rooted automorphisms is also called \emph{rooted}. We naturally identify the first-level stabilizer $\stab_1$ of $\aut\xs$ with $(\aut\xs)^\alb$, so that $(g_x)_{x\in\alb}(xv)=xg_x(v)$. Then the wreath recursion $g\mapsto\pi(g_x)_{x\in\alb}$ becomes an equality $g=\pi\cdot(g_x)_{x\in\alb}$, where $\pi$ acts on $\xs$ by the corresponding rooted automorphism.

\begin{prop} \label{pr:sierpinskibranch}
    The group  $G<\aut(\xs)$ is branch and just-infinite. More explicitly, $\rist(1)=G^\alb$, so that the wreath recursion $G\arr\Alt(\alb)\ltimes G^\alb$ is an isomorphism, and $G/G'\cong (\Z/2\Z)^4$.
\end{prop}

\begin{proof}
We have seen above that 
\[
    (ab)^2=(467)(1, 1, (ab)^2, 1, 1, 1, 1, 1).
\]
and similarly
\[(da)^2=(142)(1, 1, 1, 1, 1, 1, 1, (da)^2).\]

It follows that the subgroup $L=\langle (ab)^2, (da)^2\rangle$ acts on the set $\{1, 2, 4, 6, 7\}$ as $\Alt(5)$ and for every element $g\in L$ we have $g|_x=1$ for all $x\in\{1, 2, 4, 5, 6, 7\}$. All the other letters $\{3, 5, 8\}$ are fixed by $L$, and we have $L|_5=\{1\}$, $L|_3=\langle (ab)^2\rangle$, and $L|_8=\langle (da)^2\rangle$. It follows that $L'$ is the rooted group $\Alt(\{1, 2, 4, 6, 7\})$.

The same arguments (using other pairs of dihedral subgroups of the nucleus) show that the rooted groups $\Alt(\{1, 2, 3, 4, 5\})$, $\Alt(\{2, 3, 5, 7, 8\})$, and $\Alt(\{4, 5, 6, 7, 8\})$ are also contained in $G$. Consequently, $\Alt(\alb)<G$. 

This implies that $(12)(67)\cdot a=(1, 1, a, 1, a, 1, 1, a)\in G$. Conjugating this element by $(687)\in\Alt(\alb)<G$, we conclude that $(1, 1, a, 1, a, 1, a, 1)\in G$, hence 
\[(1, 1, a, 1, a, 1, 1, a)(1, 1, a, 1, a, 1, a, 1)=(1, 1, 1, 1, 1, 1, a, a)\in G.\]
Conjugating this element by $(567)$, we get $(1, 1, 1, 1, a, 1, 1, a)\in G$, hence 
\[(1, 1, a, 1, a, 1, 1, a)(1, 1, 1, 1, a, 1, 1, a)=(1, 1, a, 1, 1, 1, 1, 1)\in G.\]
Conjugating it by $(123)$, we get $(a, 1, \ldots, 1)\in G$.

Similar arguments show $(b, 1, \ldots, 1), (c, 1, \ldots, 1), (d, 1, \ldots, 1)\in G$, hence the elements $(g, 1, \ldots, 1)$ belong to $G$ for all $g\in G$. Conjugating them by elements of $\Alt(\alb)$, we conclude that $\rist(1)=G^{\alb}$, hence $G=\Alt(\alb)\ltimes G^{\alb}$, and the group $G$ is branch.

It follows that $\rist(v)\simeq G$ for every $v\in \xs$. Since the abelianization of $G$ is $(\Z/2\Z)^4$, Proposition~\ref{p-Grigorchuk} implies that $G$ is just-infinite.
\end{proof}

\begin{thm}\label{t-sierpinski}
The group $G$ is amenable and has Property FW.
\end{thm}

\begin{proof}
Let us show that $\limg$ has no local cut points, which will imply, by Proposition \ref{pr:sierpinskibranch} and Corollary \ref{c-intro-FW}, that $G$ has Property FW.

\begin{figure}
    \centering
    \includegraphics{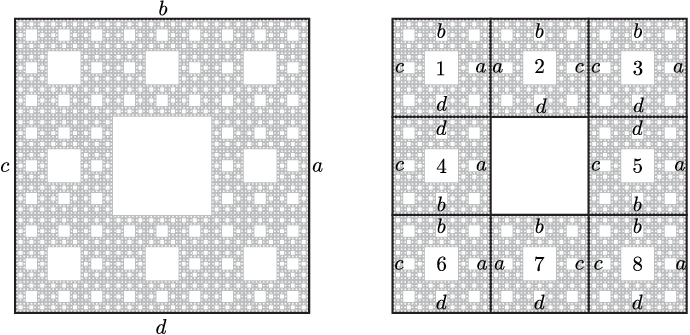}
    \caption{Sierpi\'nski carpet     \label{fig:carpet}}
\end{figure}

Let $\mathcal{C}$ be the classical Sierpi\'nski carpet. We label its sides in cyclic order by the letters $a, b, c, d$, as shown on the left-hand side of Figure~\ref{fig:carpet}. Let $\mathcal{C}_h$, for $h\in\{a, b, c, d\}$, be the (closed) side of $\mathcal{C}$ labeled by $h$. We denote by $C_{h_1h_2}$ or by $C_{h_2h_1}$, for $h_1h_2\in\{ab, bc, cd, da\}$ the intersection point of $\mathcal{C}_{h_1}$ with $\mathcal{C}_{h_2}$.

Let $\mathcal{X}$ be the space obtained by taking the quotient of the space $\mathcal{C}\times G$ by the equivalence relation generated by the identifications $(\xi, g)\sim (\xi, hg)$ for all $h\in\{a, b, c, d\}$, $\xi\in\mathcal{C}_h$, and $g\in G$. In other words, we replace the vertices of the Cayley graph of $G$ by Sierpi\'nski carpets, and paste their boundaries according to the adjacency of the vertices.

Let us describe the equivalence relation, i.e., the transitive closure of the described identifications. We have a triple $(\xi, g)\sim (\xi, h_1g)\sim (\xi, h_2h_1g)$ if and only if $\xi=C_{h_1h_2}$. It follows that the equivalence classes are 
$\{(\xi, hg)\colon h\in L_\xi\}$, where $L_\xi=\langle h_1, h_2\rangle$ if $\xi=C_{h_1h_2}$, $L_\xi=\langle h\rangle$ if $\xi$ is an interior point of the segment $\mathcal{C}_h$, and $L_\xi=\{1\}$ in all the other cases.

The group $G$ acts on $\mathcal{X}$ by the action induced by the right multiplication on $\mathcal{C}\times G$.
We denote by $\mathcal{C}\cdot g$ the image of $(\mathcal{C}, g)$ in $\mathcal{X}$, which is consistent with the action of $G$ on $\mathcal{X}$.

\begin{figure}
    \centering
    \includegraphics{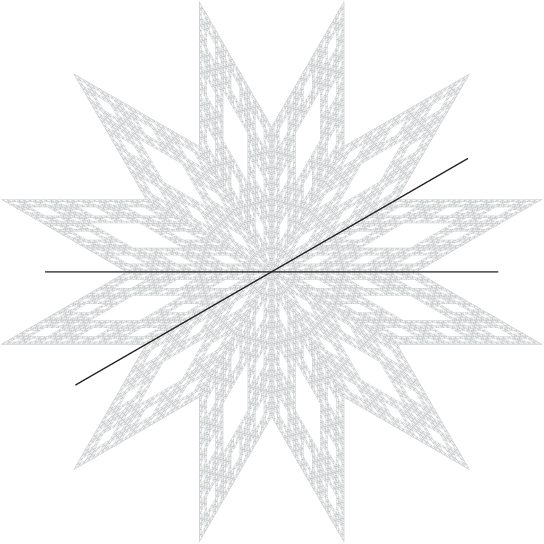}
    \caption{Neighborhood of a point with $D_6$ (dihedral of order 12) isotropy group.     \label{fig:12carpet}}
\end{figure}

The description of the equivalence relation implies that a neighborhood of a point $\xi\in \mathcal{C}\cdot g\subset \mathcal{X}$ is $\bigcup_{h\in L_\xi}\mathcal{C}\cdot hg$. See Figure~\ref{fig:12carpet}, where a neighborhood of a point $\xi=C_{h_1h_2}\cdot 1$ with $L_\xi\cong D_6$ (dihedral of order 12) is shown. Reflections with respect to the highlighted lines generate $L_\xi$ (which is the stabilizer of $C_{h_1h_2}\cdot 1$). The stabilizer of a point $\xi=C_{h_1h_2}\cdot g$ is then $g^{-1}L_\xi g$.

It follows that every point of $\mathcal{X}$ has a neighborhood homeomorphic to the Sierpi\'nski carpet, and hence $\mathcal{X}$ has no local cut points.

Consider now the space $\mathcal{X}\otimes\bim$. It follows from the definitions that it is the quotient of the space $\mathcal{C}\times\alb\times G$ by the equivalence relation generated by the identifications
\[(\xi, x, g)\sim (\xi, h(x), h|_xg)\]
for $h\in\{a, b, c, d\}$, $\xi\in\mathcal{C}_h$, $x\in\alb$, and $g\in G$.
Let us first apply the identifications for which $h|_x=1$. It follows from the wreath recursions (see the diagram in Figure \ref{fig:carpetmoore}) that for every fixed $g\in G$ we will paste together eight sets $(\mathcal{C}, x, g)$, for $x\in\alb$ as shown on the right-hand side of Figure~\ref{fig:carpet} (compare it with Figure~\ref{fig:carpetmoore}). The digits $x\in\{1, 2, \ldots, 8\}$ label the corresponding sets $(\mathcal{C}, x, g)$. Hence, the connected components we get after these identifications will be sets similar to $\mathcal{C}$ with ratio 3. The remaining identifications, producing $\mathcal{X}\otimes\bim$, will agree with the pasting rules used to define $\mathcal{X}$.

Consequently, we get a $G$-equivariant similarity $\mathcal{X}\otimes\bim\arr\mathcal{X}$ contracting all distances exactly by coefficient $1/3$. Theorem~\ref{th:uniquenessoflimg} implies that $\mathcal{X}$ is equivariantly homeomorphic to the limit $G$-space $\limg$. It was proved in~\cite[\S6.3]{MNZ} that the group $G$ is amenable.
\end{proof}

	\bibliographystyle{alpha}
	\bibliography{biblio}
	
\bigskip

{\small

\noindent\textit{Nicol\'as Matte Bon\\
Université Lyon 1, Centrale Lyon, INSA Lyon, Université Jean Monnet, CNRS, ICJ UMR5208, 69622 Villeurbanne, France.\\}
\href{mailto:mattebon@math.univ-lyon1.fr}{mattebon@math.univ-lyon1.fr}

\smallskip

\noindent\textit{Volodymyr Nekrashevych\\
	Department of Mathematics\\
 Texas A\&M University, College Station, TX 77843-3368, United States of America\\}
\href{mailto:nekrash@math.tamu.edu}{nekrash@math.tamu.edu}

\smallskip

\noindent\textit{Tianyi Zheng\\
	Department of Mathematics\\
 University of California, San Diego (UCSD),
9500 Gilman Drive \# 0112,
La Jolla, CA  92093-0112, United States of America\\}
\href{mailto:tzheng2@math.ucsd.edu}{tzheng2@math.ucsd.edu}

}

\end{document}